\documentclass[12pt]{amsart}
\usepackage{amsmath,amsthm,amscd,amsfonts,amssymb,epic,eepic,bbm, mathabx,graphicx,ytableau}
\allowdisplaybreaks

\newtheorem{thm}{Theorem}[section]
\newtheorem{cor}[thm]{Corollary}
\newtheorem{conj}[thm]{Conjecture}
\newtheorem{lem}[thm]{Lemma}
\newtheorem{defn}[thm]{Definition}

\newtheorem{rem}[thm]{Remark}
\newtheorem{example}[thm]{Example}
\theoremstyle{remark}

\newcommand{\la}{\lambda}

\begin{document}
\nocite{DW,Gasha,GS,Guay,Hai,HP,SWE,SW,S3,S1,S2,SS}

\title[$e$-positivity and $e$-unimodality of chromatic quasisymmetric functions]{On $e$-positivity and $e$-unimodality of chromatic quasisymmetric functions}

\author{Soojin Cho}
\address{Department of Mathematics, Ajou University, Suwon 16499 Republic of Korea}
\email{chosj@ajou.ac.kr}

\author{JiSun Huh}
\address{Department of Mathematics, Ajou University, Suwon 16499 Republic of Korea}
\email{hyunyjia@ajou.ac.kr}


\thanks{This work is supported by the Basic Science Research Program through the National Research Foundation of Korea (NRF) funded by the Ministry of Education (NRF-2015R1D1A1A01057476).}

\begin{abstract} The $e$-positivity conjecture and the $e$-unimodality conjecture of chromatic quasisymmetric functions are proved for some classes of natural unit interval orders. Recently, J. Shareshian and M. Wachs introduced chromatic \emph{quasisymmetric} functions as a refinement of Stanley's chromatic symmetric functions and conjectured the $e$-positivity and the $e$-unimodality of these functions. The $e$-positivity of chromatic quasisymmetric functions implies the $e$-positivity of corresponding chromatic symmetric functions, and our work resolves Stanley's conjecture on chromatic symmetric functions of $(3+1)$-free posets for two classes of natural unit interval orders.
\end{abstract}

\maketitle

\section{Introduction}

In 1995, R. Stanley \cite{S1} introduced the {\it chromatic symmetric function} $X_G(\bold{x})$ associated with any simple graph $G$, which generalizes the chromatic polynomial $\chi_G(n)$ of $G$. There have been great deal of researches on these functions in various directions and purposes, and
one of the long standing and well known conjecture due to Stanley on chromatic symmetric functions states that a chromatic symmetric function of any $(3+1)$-free poset is a linear sum of elementary symmetric function basis $\{e_\la\}$ with \emph{nonnegative} coefficients. 
Recently, Shareshian and Wachs \cite{SW} introduced a chromatic \emph{quasi}symmetric refinement $X_G({\bf x}, t)$ of  chromatic symmetric function $X_G({\bf x})$ for a graph $G$. They conjectured the $e$-positivity and the $e$-unimodality of the chromatic quasisymmetric functions of natural unit interval orders: That is, if $X_G({\bf x}, t)=\sum_{j=0}^m a_j({\bf x}) t^j$ for  the incomparability graph $G$ of a natural unit interval order, then $a_j({\bf x})$, $0\leq j\leq m$, is a nonnegative linear sum of $e_\la$'s, which is a refinement of Stanley's $e$-positivity conjecture. Moreover, $e$-unimodality conjecture states that $a_{j+1}({\bf x})-a_j({\bf x})$ is a sum of $e_\la$'s with nonnegative coefficients for $0\leq j<\frac{m-1}{2}$.

Our work in this paper is to give combinatorial proofs of the conjectures on $e$-positivity and $e$-unimodality of chromatic quasisymmetric functions for certain (distinct) families of natural unit interval orders. We use the Schur function expansion of chromatic quasisymmetric functions in terms of Gasharov's $P$-tableaux (\cite{Gasha}), that was done by Shareshian and Wachs in \cite{SW}. For the $e$-positivity, we use Jacobi-Trudi expansion of Schur functions into elementary symmetric functions to write chromatic quasisymmetric functions as a sum of elementary symmetric functions with coefficients of signed sum of positive $t$ polynomials. We then define a sign reversing involution to cancel out negative terms and obtain only positive terms. We use the combinatorial model for the coefficients of $e$-expansion of $X_G({\bf x}, t)$ we obtained and an inductive argument on $P$-tableaux to find explicit $e$-expansion formulae of $X_G({\bf x}, t)$ and this shows the $e$-unimodality of $X_G({\bf x}, t)$ for some classes of natural unit interval orders.

Our proof for $e$-positivity of $X_G({\bf x}, t)$ proves Stanley's $e$-positivity conjecture as we mentioned. We have to remark that one class of natural unit interval orders we consider (in Section~\ref{sec:firstmain})
was considered by Stanley and Stembridge and the $e$-positivity of the corresponding chromatic symmetric functions was proved (Remark 4.4 in \cite{SS}). Our proof method, however, is different from Stanley-Stembridge's and gives a new proof. We recently notice that there is another (new) $e$-positivity proof of the chromatic symmetric functions for the same class by Harada and  Precup \cite{HP}.

Two families we consider for the $e$-positivity of the chromatic quasisymmetric functions include $(2^{n-1}-n)$ and $\frac{(n-3)(n-4)}{2}$ natural unit interval orders, respectively, out of $\frac{1}{n}{{2(n-1)}\choose {n-1}}$ natural unit interval orders that are corresponding to \emph{connected} graphs on $n$ vertices. We show the $e$-unimodality of the chromatic quasisymmetric functions for $(n-2)^2$ orders out of $\frac{1}{n}{{2(n-1)}\choose {n-1}}$ natural unit interval orders.
  
The present paper is organized as follows. After we introduce some important notions and known background works in Section \ref{sec:pre}, we provide two classes of natural unit interval orders satisfying the refined $e$-positivity conjecture in Section \ref{sec:main}. In Section \ref{sec:last} we consider some special natural unit interval orders among the ones that were considered in Section \ref{sec:main}, and obtain the explicit $e$-basis expansion of the corresponding chromatic quasisymmetric functions, which shows the $e$-unimodality of them.

\vspace{3mm}
\section{Preliminaries}\label{sec:pre}

In this section, we define necessary notions, setup notations and review known related results that will be used to develop our arguments.

\vspace{2mm}
\subsection{Chromatic quasisymmetric functions and the positivity conjecture.}
We let $\mathbb{P}$ be the set of positive integers and use $[n]$ to denote the set $\{1, 2, \dots, n\}$ for $n\in \mathbb P$. For a positive integer $n$, a \emph{partition} of $n$ is a sequence $\la=(\la_1, \la_2, \dots, \la_\ell)$ of positive integers such that $\la_i\geq \la_{i+1}$ for all  $i$ and $\sum_i \la_i =n$ and we use the notation $\la\vdash n$ to denote that $\la$ is a partition of $n$. For a partition $\la$, $\la'=(\la'_1, \dots, \la'_{\la_1})$ is the \emph{conjugate} of $\la$ defined as $\la'_j=|\{ i\,|\,\la_i\geq j\}|$. We let $\bold{x}=(x_1,x_2,\dots)$ be a  sequence of commuting variables. For $n\in \mathbb{P}$, the \emph{$n^{\mbox{\tiny{th}}}$ elementary symmetric function} $e_n$ is defined as 
$e_n=\sum_{i_1<\cdots<i_n} x_{i_1}\cdots x_{i_n}\,,$ and the \emph{$\mathbb{Q}$-algebra $\Lambda_{\mathbb{Q}}$ of symmetric functions} is the subalgebra of $\mathbb{Q}[[x_1, x_2, \dots]]$ generated by the $e_n$'s;
$\Lambda_{\mathbb{Q}}=\Lambda=\mathbb{Q}[e_1, e_2, \dots]\,.$ Then $\Lambda=\bigoplus_{n=0}^\infty \Lambda^n$ where $\Lambda^n$ is the subspace of symmetric functions of degree $n$, and $\{e_\la \,|\, \la\vdash n\}$ is a basis of $\Lambda^n$ where $e_\la=e_{\la_1}\cdots e_{\la_\ell}$ for $\la=(\la_1, \la_2, \dots, \la_\ell)$. The set of \emph{Schur functions} $\{s_\la\,|\,\la\vdash n\}$ forms another well known basis of $\Lambda^n$, where $s_\la=\mbox{det}[e_{\la'_i-i+j}]$ is a determinant of a $\la_1\times \la_1$ matrix, that is called \emph{Jacobi-Trudi identity}. 

We now give the definition of \emph{chromatic quasisymmetric functions} that is the main object of the current work. Recall that a {\it proper coloring} of a simple graph $G=(V,E)$ with vertex set $V$ and edge set $E$ is any function $\kappa : V \rightarrow \mathbb{P}$
satisfying $\kappa(u)\neq \kappa(v)$ for any $u, v\in V$ such that $\{u,v\}\in E$. 

\begin{defn}[Shareshian-Wachs \cite{SW}] 
For a simple graph $G=(V,E)$ which has a vertex set $V\subset\mathbb{P}$, the \emph{chromatic quasisymmetric funcion} of $G$ is
$$X_G(\bold{x},t)=\sum_{\kappa}t^{\rm{asc}(\kappa)}\bold{x}_{\kappa},$$
where the sum is over all proper colorings $\kappa$ and
$${\rm asc}(\kappa)=\{\{i,j\}\in E~|~i<j~\text{and}~\kappa(i)<\kappa(j)\}.$$
\end{defn}

Chromatic quasisymmetric function $X_G(\bold{x},t)$ is defined as a refinement of Stanley's \emph{chromatic symmetric function} $X_G(\bold{x})$ introduced in \cite{S1};
$$X_G(\bold{x},t)|_{t=1}=X_G(\bold{x})\,.$$

Before we state a conjecture on (quasi)chromatic symmetric functions, we give some necessary definitions first. The {\it incomparability graph} ${\rm inc}(P)$ of a poset $P$ is a graph which has as vertices the elements of $P$, with edges connecting pairs of incomparable elements.  
The posets we are interested in are \emph{natural unit interval orders}. There are many equivalent descriptions of natural unit interval orders (see Section 4 of \cite{SW}) and we use the following definition for our work.

\begin{defn}[\cite{SW}]
Let $\bold{m}:=(m_1, m_2, \dots, m_{n-1})$ be a list of positive integers satisfying $m_1 \leq m_2 \leq \cdots \leq m_{n-1} \leq n$ and $m_i \geq i$ for all $i$. The corresponding \emph{natural unit interval order} $P(\bold{m})$ is the poset on $[n]$ with the order relation given by $i<_{P(\bold{m})} j$ if $i<n$ and $j\in \{m_i +1, m_i +2, \dots, n\}$.
\end{defn}

It is easy to see that the incomparability graph  ${\rm inc}(P)$ of $P=P(m_1,\dots,m_{n-1})$ is the graph with vertex set $[n]$ and $\{i, j\}$, $i<j$, is an edge of ${\rm inc}(P)$ if and only if $j\leq m_i$, and if $i<_P j$ then $i<j$. Note that Catalan number $C_n=\frac{1}{n+1}\binom{2n}{n}$ counts the natural unit interval orders with $n$ elements. 

\begin{example} \label{ex:stair}
For each $r\in [n]$, define $P_{n,r}$ to be the poset on $[n]$ with order relation given by $i<_{P_{n,r}}j$ if $j-i\geq r$. Then the poset $P_{n,r}$ is the natural unit interval order $P(r,r+1,r+2,\dots,n,\dots,n)$ whose incomparability graph $G_{n,r}$ is the graph with vertex set $[n]$ and edge set $\{\{i,j\}~|~0<j-i<r\}$.
\end{example}

Shareshian and Wachs showed that if $G$ is the incomparability graph of a natural unit interval order then $X_G(\bold{x},t)$ is a polynomial with very nice properties:

\begin{thm}[Theorem 4.5 and Corollary 4.6 in \cite{SW}]\label{thm:SW}
If $G=(V, E)$ is the incomparability graph of a natural unit interval order then the coefficients of $t^i$ in $X_G(\bold{x},t)$ are symmetric functions and form a palindromic sequence in the sense that $X_G(\bold{x},t)=t^{|E|}X_G(\bold{x},t^{-1})$.
\end{thm}

Shareshian and Wachs also made a conjecture on the $e$-positivity and the $e$-unimodality of $X_G(\bold{x},t)$. Remember that one says a symmetric function $f({\bf{x}})\in\Lambda^n $ is $b$-{\it positive} if the expansion of $f(\bf{x})$ in the basis $\{b_{\la}\}$ has nonnegative coefficients when $\{b_\la\,|\,\la\vdash n\}$ is a basis of $\Lambda^n$.

\begin{conj}[Shareshian-Wachs \cite{SW}]\label{conj:SW}
If $G$ is the incomparability graph of a natural unit interval order, then $X_G(\bold{x},t)$ is $e$-positive and $e$-unimodal. That is, if $X_G(\bold{x},t)=\sum_{i=0}^{m}a_i(\bold{x})t^i$ then $a_{i}(\bold{x})$ is $e$-positive for all $i$, and $a_{i+1}(\bold{x})-a_{i}(\bold{x})$ is $e$-positive whenever $0\leq i < \frac{m-1}{2}$.
\end{conj}

A finite poset $P$ is called $(r+s)$-{\it free} if $P$ does not contain an induced subposet that is isomorphic to the direct sum of an $r$ element chain and an $s$ element chain. Due to the result by Guay-Paquet \cite{Guay}, Conjecture~\ref{conj:SW} specializes to the famous $e$-positivity conjecture on the chromatic symmetric functions of Stanley and Stembridge:

\begin{conj}[Stanley-Stembridge \cite{S1,SS}]\label{conj:SS} If a poset $P$ is $(3+1)$-free, then $X_{{\rm inc}(P)}(\bold{x})$ is $e$-positive. 
\end{conj}

Since $e_\la$ is $s$-positive for all $\la$, Conjecture~\ref{conj:SW} implies the $s$-positivity of $X_G(\bold{x},t)$ and they were shown to be true by Gasharov when $t=1$ and by Shareshian-Wachs for general case using $P$-tableaux.

\begin{defn}[Gasharov \cite{Gasha}]
Given a poset $P$ with $n$ elements and a partition $\la$ of $n$, a \emph{$P$-tableau of shape} $\la$ is a filling $T=[a_{i,j}]$ of a Young diagram of shape $\la$ in English notation \\

\begin{center}
$T$=\begin{ytableau}      
a_{1,1}&a_{1,2}&\cdots\cr
a_{2,1}&a_{2,2}&\cdots\cr
a_{3,1}&\vdots \cr
\vdots\cr
\end{ytableau}\\
\end{center}
with all elements of $P$ satisfying the following properties:
\begin{enumerate}
\item Each element of $P$ appears exactly once.
\item For all $i$ and $j$, $a_{i,j}<_Pa_{i,j+1}$.
\item For all $i$ and $j$, $a_{i+1,j}\not <_P a_{i,j}$.
\end{enumerate}
\end{defn}

\begin{defn}[Shareshian-Wachs \cite{SW}] For a finite poset $P$ on a subset of $\mathbb P$ and a $P$-tableau $T$, let $G={\rm inc}(P)$ be the incomparability graph of $P$. Then an edge $\{i,j\}\in E(G)$ is a \emph{$G$-inversion of $T$} if $i<j$ and $i$ appears below $j$ in $T$. We let $\emph{inv}_G(T)$ be the number of $G$-inversions of $T$.
\end{defn}

\begin{example}
Let $P=P(3,5,5,6,7,8,8)$ be a natural unit interval order and $G$ be the incomparability graph of $P$. Then 
\begin{center}
$T$= \begin{ytableau}      
1&4&7\cr
3&6\cr
2&8\cr
5\cr
\end{ytableau}\\
\end{center}
is a $P$-tableau of shape $(3,2,2,1)$ and
$${\rm inv}_G(T)=|\{\{2,3\},\{2,4\},\{3,4\},\{5,6\},\{5,7\},\{6,7\}\}|=6.$$
Note that $\{2,3\}\in E(G)$ since $2\not <_P 3$, and $2<3$.
\end{example}

The following theorem is the $s$-positivity result of Shareshian-Wachs for the chromatic quasisymmetric functions of natural unit interval orders, which specializes to the Gasharov's $s$-positivity result for chromatic symmetric functions of $(3+1)$-free posets.

\begin{thm}[\cite{SW},\cite{Gasha}]\label{thm:Gasha} 
Let $P$ be a natural unit interval order and let $G$ be the incomparability graph of $P$. Then
$$X_G(\bold{x},t)=\sum_{T} t^{{\rm inv}_G(T)}s_{\la(T)},$$
where the sum is over all $P$-tableaux and $\la(T)$ is the shape of $T$, and therefore, $X_G(\bold{x},t)$ is $s$-positive.
\end{thm}

\vspace{1mm}
We will start our argument in Section~\ref{sec:main} and Section~\ref{sec:last}, from Theorem~\ref{thm:Gasha} to deal with the $e$-positivity and the $e$-unimodality in Conjecture~\ref{conj:SW}. We state a simple lemma due to Stanley and make some useful remarks.

\vspace{1mm}
\begin{lem}[\cite{S3}]\label{lem:basic} 
If  $A(t)$  and $B(t)$ are unimodal and palindromic polynomials with nonnegative coefficients and centers of symmetry $m_A$, $m_B$ respectively, then  $A(t)B(t)$ is unimodal and palindromic with nonnegative coefficients and center of symmetry $m_A+m_B$.
\end{lem}

\vspace{1mm}
\begin{rem}\label{rem:basic} 
Let $G$ be the incomparability graph of a natural unit interval order and let
$X_G(\bold{x},t)=\sum_{i=0}^{m} a_i(\bold{x})\, t^i= \sum_{\la\vdash n} C_\la(t)\, e_\la(\bold{x})$. Then
\begin{enumerate}
\item $X_G(\bold{x},t)$ is palindromic in $t$ with center of symmetry $\frac{m}{2}$ if and only if $C_\la(t)$ is a palindromic polynomial with center of symmetry $\frac{m}{2}$.
\item $X_G(\bold{x},t)$ is $e$-positive if and only if $C_\la(t)$ is a polynomial with nonnegative coefficients for all $\la$.
\item $X_G(\bold{x},t)$ is $e$-unimodal with center of symmetry $\frac{m}{2}$ if every $C_\la(t)$ is unimodal with the same center of symmetry $\frac{m}{2}$. 
\end{enumerate}
\end{rem} 

\vspace{1mm}
Shareshian and Wachs obtained an explicit formula for the coefficient $C_n(t)$ of $e_{n}$ in the $e$-basis expansion of $X_{G}(\bold{x},t)$, which shows the positivity and the unimodality of $C_n(t)$. 

\vspace{1mm}
\begin{thm}[\cite{SW}]\label{thm:en}
Let $G$ be the incomparability graph of a natural unit interval order with $n$ elements and let $X_G(\bold{x},t)= \sum_{\la\vdash n} C_\la(t)\, e_\la(\bold{x})$. Then  
$$C_n(t)=[n]_t \prod_{i=2}^{n} [b_i]_t$$ 
and is therefore positive and unimodal with center of symmetry $\frac{|E(G)|}{2}$, where 
$[n]_t=1+t+t^2+\cdots+t^{n-1}$
and
$b_i=|\{ \{j,i\}\in E(G)~|~j<i\}|.$  
\end{thm}

\vspace{1mm}
\begin{rem} \label{rem:uni} 
We summarize known results on $e$-positivity and $e$-unimodality of $X_{{\rm inc}(P)}(\bold{x}, t)$ for natural unit interval orders $P=P(\bold{m})$ with $\bold{m}=(m_1,m_2,\dots,m_{n-1})$.
\begin{enumerate}
\item Let $\bold{m}=(2,3,\dots,n)$. The incomparability graph of $P$ is the path $G_{n,2}=1-2-\cdots-n$. Stanley \cite{SWE} obtained the formula for the generating function and Haiman \cite{Hai} showed that $X_{G_{n,2}}(\bold{x},t)$ is $e$-positive and $e$-unimodal. 
\item Let $\bold{m}=(m_1, m_2,n,\dots,n)$. Shareshian and Wachs \cite{SW} obtained an $e$-positive and $e$-unimodal explicit formula of $X_{{\rm inc}(P)}(\bold{x},t)$. 
\item For a composition $\alpha=(\alpha_1,\alpha_2,\dots, \alpha_k)$ of $n+k-1$, let $m_1=\cdots=m_{\alpha_1-1}=\alpha_1$ and $m_i=\sum_{j=1}^{\ell}(\alpha_j-1)+1$ if $\sum_{j=1}^{\ell-1}(\alpha_j-1)+1 \leq i \leq \sum_{j=1}^{\ell}(\alpha_j-1)$ for $\ell=2,3,\dots,k$. The incomparability graph of $P$ is the {\it $K_{\alpha}$-chain}, the concatenate of complete graphs whose sizes are given by the parts of $\alpha$. 
Gebhard and Sagan \cite{GS} proved that $X_{{\rm inc}(P)}(\bold{x})$ is $e$-positive.  
\item Let $\bold{m}=(2,3,\dots, m+1,n,\dots,n)$. The incomparability graph of $P$ is a {\it lollipop graph} $L_{m,n-m}$ that is obtained by connecting the vertex $m$ of the path $1-2-\cdots-m$ and the vertex $m+1$ of the complete graph on $\{m+1,m+2,\dots,n\}$ as an edge. Dahlberg and van Willigenburg \cite{DW} computed an explicit $e$-positive formula for $X_{L_{m,n-m}}(\bold{x})$. 
\item Let $\bold{m}=(r, m_2, m_3 \dots,m_r, n,\dots,n)$. The incomparability graph of $P$ is a complement graph of a bipartite graph. Stanley and Stembridge \cite{SS} proved that $X_G(\bold{x})$ is $e$-positive. 
\end{enumerate}
\end{rem}

\vspace{2mm}
\subsection{Basic setup}\label{sec:nui}
We often use Catalan paths of length $n$ from $(0,0)$ to $(n,n)$ to represent natural unit interval orders of size $n$: The corresponding Catalan path of $P=P(m_1, m_2, \dots, m_{n-1})$ has the $i$th horizontal step on the line $y=m_i$ for all $i=1,2,\dots,n-1$ and the last horizontal step on the line $y=n$. Then the first $m_1$ vertical steps are on the line $x=0$ and the $i$th vertical step for $m_1<i\leq n$ is on the line $x=k$, where $k$ is the largest element $j$ such that $\{j,i\}\not\in E({\rm inc}(P))$ with $j < i$. Figure \ref{Fig1} shows the corresponding Catalan path of $P(3,3,4,6,6)$ and its incomparability graph. 

\begin{figure}[h]
\centering {\includegraphics[height=3.3cm]{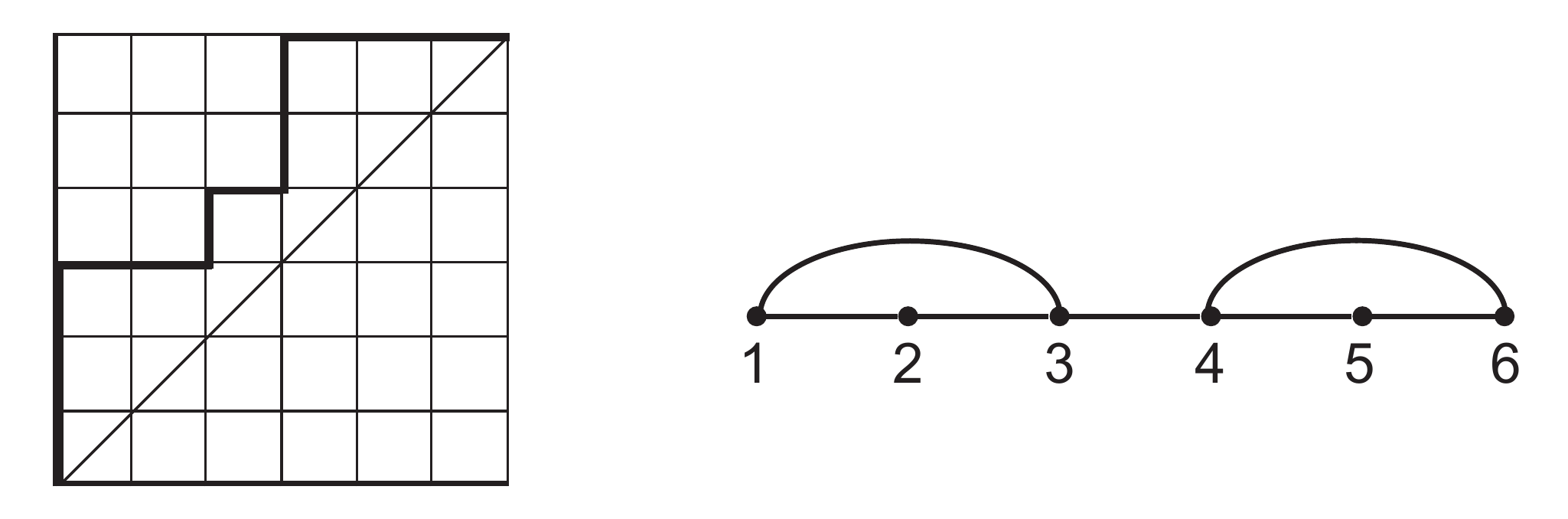} }
\caption{The corresponding Catalan path of $P=P(3,3,4,6,6)$ and ${\rm inc}(P)$. 
}\label{Fig1}
\end{figure}

Following lemma is immediate from the definitions.

\begin{lem} Let $P=P(m_1,m_2,\dots,m_{n-1})$ be a  natural unit interval order and $C$ be the corresponding Catalan path. Then, the incomparability graph $G={\rm inc}(P)$ is connected if and only if $m_i>i$ for all $i=1, \dots, n-1$ or equivalently, $C$ meets the line $y=x$ only at $(0,0)$ and $(n,n)$.
\end{lem}

If a finite simple graph $G$ is a disjoint union of subgraphs $G_1, \dots, G_\ell$, then $$X_{G}(\bold{x},t)=\prod_{i=1}^\ell X_{G_{i}}(\bold{x},t)\,.$$ Hence, we may restrict our attention to natural unit interval orders $P(m_1,m_2,\dots, m_{n-1})$ with $m_i\neq i$ for all $i=1,2,\dots,n-1$. The corresponding Catalan paths of these unit interval orders are the {\it prime} Catalan paths which meet the diagonal line $y=x$ only on the two end points. There are $C_{n-1}$ prime Catalan paths of length $n$.

Given a Catalan path $C$ of length $n$, the {\it bounce path} of $C$ is the path you travel in the following way: Starting at $(0,0)$, travel north along $C$ until you encounter the beginning of an east step of $C$. Then turn east and travel straight until you hit the main diagonal $y=x$ and turn north to travel until you again encounter the beginning of an east step of $C$.Then turn east and travel to the diagonal, and so on. Continue in this way until you arrive at $(n,n)$. The \emph{bounce number} of $C$ is the number of connected regions that the bounce path of $C$ and the line $y=x$ enclose; that is one less than the number of times the bounce path of $C$ meets the line $y=x$. For example, the bounce number of the Catalan path corresponding to $P(3, 3, 4, 6, 6)$ is $2$.

The followings are important lemmas that we will use in Sections~\ref{sec:main} and \ref{sec:last}.

\vspace{1mm}
\begin{lem} \label{lem:chain}
Let $C$ be the corresponding Catalan path of a natural unit interval order $P$ and let $r$ be the bounce number of $C$. If $$X_{{\rm inc}(P)}(\bold{x},t)=\sum_{\la}B_{\la}(t)s_{\la},$$ 
then $B_{\la}(t)=0$ for partitions $\la=(\la_1, \dots, \la_k)$ with $\la_1>r$. 
\end{lem} 
\begin{proof}
Suppose that the bounce path of $C$ meets the line $y=x$ at $r+1$ vertices,
$$(0,0)=(b_0,b_0), ~(b_1,b_1), ~(b_2,b_2), ~\dots, ~(b_{r-1},b_{r-1}), ~(b_r,b_r)=(n,n),$$ 
with $0=b_0<b_1<\cdots<b_{r-1}<b_r=n$ and let $I_k=\{x \in P ~|~ b_{k-1}<x\leq b_k\}$ for $k=1,2,\dots,r$. Note that, if $i<_Pj$ for $i\in I_k$ and $j \in I_{\ell}$, then $k<\ell$ since $i$ and $j$ in the same set $I_k$ are incomparable. Hence, $P$ has no chain with more than $r$ elements and there is no $P$-tableau $T$ of shape $\la$ with $\la_1>r$. Since  $B_{\la}(t)=\sum t^{{\rm inv}_G(T)}$ where the sum is over all $P$-tableaux of shape $\la$ by Theorem \ref{thm:Gasha}, this shows that $B_{\la}(t)=0$ if $\la_1>r$.
\end{proof}

Palindromic property of $X_G(\bold{x},t)$ in Theorem~\ref{thm:SW} implies the following lemma:

\vspace{1mm}
\begin{lem} \label{lem:reflection}
Given a natural unit interval order $P$ with $n$ elements, let $C$ be the corresponding Catalan path of $P$. If $\widetilde{C}$ is the reflection of the Catalan path $C$ about the line $y=n-x$, then 
$$X_{G}(\bold{x},t)=X_{\widetilde{G}}(\bold{x},t),$$ 
where $\widetilde{P}$ is the corresponding natural unit interval order of the Catalan path $\widetilde{C}$ and $G={\rm inc}(P)$, $\widetilde{G}={\rm inc}(\widetilde{P})$ are incomparability graphs. 
\end{lem}
\begin{proof}
For a natural unit interval order $P=P(m_1,m_2, \dots,m_{n-1})$ and its incomparability graph $G$, the corresponding Catalan path $C$ of $P$ has the $i$th horizontal step on the line $y=m_i$ and $|\{\{i,j\}\in E(G)~|~i<j \}|=m_i-i$ for $i=1,2,\dots,n-1$. Since  $\widetilde{C}$ is the reflection of $C$, it is again a Catalan path and has the $(n-i+1)$st vertical step on the line $x=n-m_i$. Note that, by the construction of the corresponding Catalan path, $n-m_i$ is the largest element $j$ such that $\{j,n-i+1\}\not\in E(\widetilde{G})$ with $j<n-i+1$. Hence, for a fixed $i$, $|\{\{j,n-i+1\}\in E(\widetilde{G})~|~j<n-i+1 \}|=(n-i)-(n-m_i)=m_i-i$.\\
Now, suppose that $C$ has its $i$th vertical step on the line $x=k$. Since $k$ is the largest element $j$ such that $\{j,i\}\not\in E(G)$ with $j<i$, $|\{\{j,i\}\in E(G)~|~j<i\}|=(i-1)-k$ and the $(n-i+1)$st horizontal step in $\widetilde{C}$ is on the line $y=n-k$. Therefore, we have $|\{\{n-i+1,j\}\in E(\widetilde{G})~|~n-i+1<j\}|=(n-k)-(n-i+1)=i-k-1$.

We therefore can conclude that if we replace the vertex $i$ with $n-i+1$ for all $i\in V(G)$, then we obtain the graph $\widetilde{G}$. Let $\widetilde{\kappa}:V(\widetilde{G})\rightarrow \mathbb{P}$ be the proper coloring satisfying $\widetilde{\kappa}(i)=\kappa(n-i+1)$ for a proper coloring $\kappa : V(G)\rightarrow \mathbb{P}$, then $\bold{x}_{\kappa}=\bold{x}_{\widetilde{\kappa}}$ and ${\rm asc}(\kappa)+{\rm asc}(\widetilde{\kappa})=|E(\widetilde{G})|$. Thus, we have
$$X_{G}(\bold{x},t)=\sum_{\kappa}t^{{\rm asc}(\kappa)} \bold{x}_{\kappa}=\sum_{\widetilde{\kappa}}t^{|E(\widetilde{G})|-{\rm asc}(\widetilde{\kappa})} \bold{x}_{\widetilde{\kappa}}=t^{|E(\widetilde{G})|}\sum_{\widetilde{\kappa}}(t^{-1})^{{\rm asc}(\widetilde{\kappa})} \bold{x}_{\widetilde{\kappa}}.$$
Finally, the palindromic property of $X_{\widetilde{G}}(\bold{x},t)$ stated in Theorem~\ref{thm:SW} gives that
$$X_{\widetilde{G}}(\bold{x},t)=t^{|E(\widetilde{G})|}X_{\widetilde{G}}(\bold{x},t^{-1})=X_{G}(\bold{x},t)$$
as we desired. 
\end{proof}

\vspace{1mm}
\begin{example} For a natural unit interval order $P=P(3,3,4)$, let $C$ be the corresponding Catalan path. Then $\widetilde{P}=P(2,4,4)$ is the natural unit interval order whose corresponding Catalan path $\widetilde{C}$ is the reflection of $C$. Figure \ref{Fig2} shows $C$, $G={\rm inc}(P)$,  $\widetilde{C}$, and $\widetilde{G}={\rm inc}(\widetilde{P})$. 

\begin{figure}[h]
\centering {\includegraphics[height=3cm]{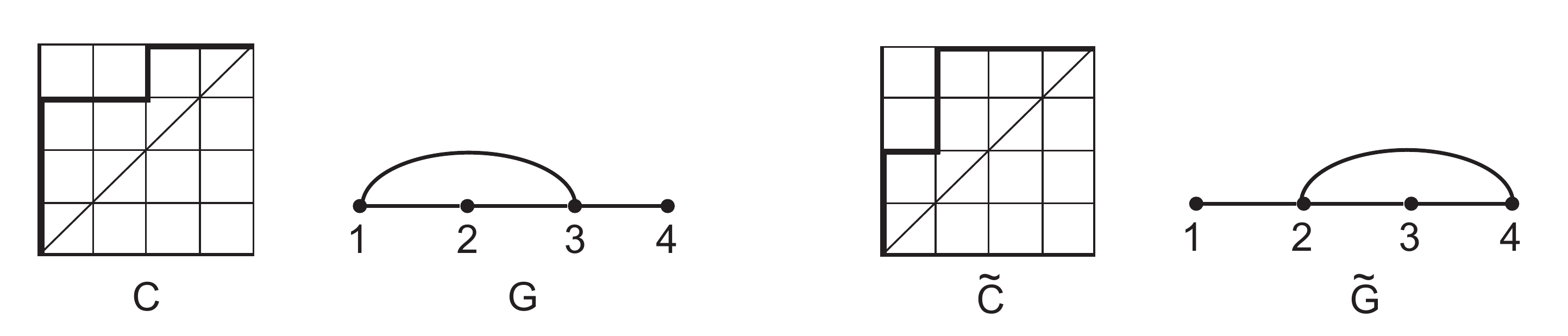} }
\caption{Catalan path $C$, its reflection $\widetilde{C}$, and their incomparability graphs}\label{Fig2}
\end{figure}

\noindent Then the chromatic quasisymmetric functions of $G$ and  of $\widetilde{G}$ are the same; (see Theorem \ref{thm:form1})
$$X_{G}(\bold{x},t)=[2]_t([4]_te_4+t[2]_te_{(3,1)}) =X_{\widetilde{G}}(\bold{x},t).$$
\end{example}

Lemma \ref{lem:reflection} will play an important role in Section~\ref{sec:last} when we deal with the $e$-unimodality.

\vspace{3mm}

\section{Two $e$-positive classes of natural unit interval orders} \label{sec:main}

In this section we provide two classes of natural unit interval orders $P=P(\bold{m})$ satisfying that $X_{{\rm inc}(P)}(\bold{x},t)$ is $e$-positive. One of them is the class of natural unit interval orders mentioned in Remark \ref{rem:uni} (5).

\vspace{1mm}
\subsection{The first $e$-positive class} \label{sec:firstmain}

In this subsection we consider natural unit interval orders $P=P(r,m_2, m_3, \dots, m_{r}, n, \dots, n)$ of size $n$ for a positive integer $r<n$. Let $G$ be the incomparability graph of $P$. Then induced subgraphs of $G$ on $\{1,2,\dots,r\}$ and $\{r+1,r+2, \dots,n\}$ are complete graphs. The first figure in Figure \ref{Fig3} shows a corresponding Catalan path (solid line) of $P$. The dotted line indicates the bounce path and the bounce number of the natural unit interval orders we consider is \emph{two}. The second figure in Figure \ref{Fig3} shows an example of $G={\rm inc}(P)$ with $P=(4,4,5,6,9,9,9,9)$. We can see that $G$ has two complete graphs as subgraphs.

\begin{figure}[h]
\centering {\includegraphics[height=4.3cm]{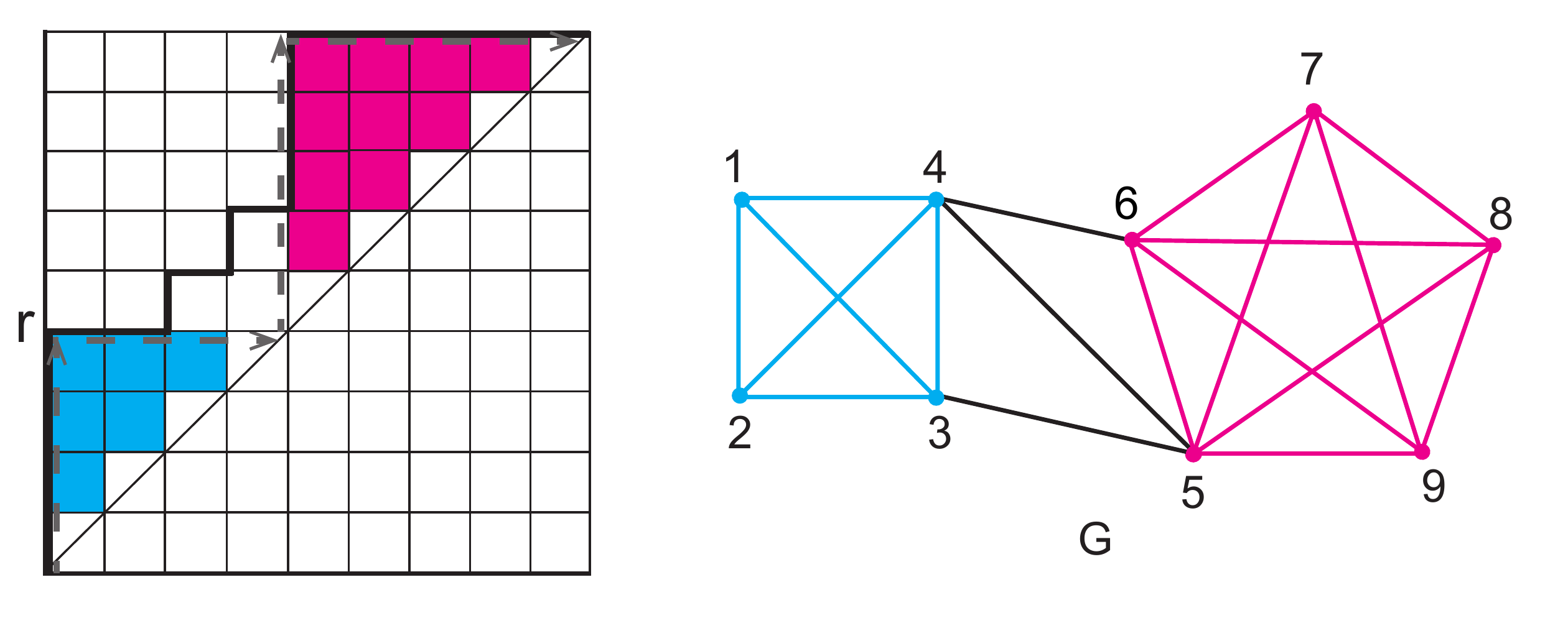} }
\caption{A Catalan path and the bounce path of $P(r,m_2, m_3, \dots, m_{r}, n, \dots, n)$}\label{Fig3}
\end{figure}

If we let
$$X_{G}(\bold{x},t)=\sum_{\la \vdash n}B_{\la}(t)s_{\la},$$
then by Lemma \ref{lem:chain}, $B_{\la}(t)\neq 0$ only for $\la$'s whose conjugate $\la '$ has only two parts. We let $k$ be the largest positive integer satisfying $B_{2^k1^{n-2k}}(t)\neq 0$. Then, due to Jacobi-Trudi identity, we have
\begin{align*}
s_{1^n}&=e_n \\
s_{2^11^{n-2}}&=e_{(n-1,1)}-e_n\\
s_{2^21^{n-4}}&=e_{(n-2,2)}-e_{(n-1,1)}\\
      &~~~\vdots\\
s_{2^k1^{n-2k}}&=e_{(n-k,k)}-e_{(n-k+1,k-1)},             
\end{align*}
Hence, we can derive the $e$-basis expansion of $X_{G}(\bold{x},t)$ as follows.
\begin{eqnarray*}
X_{G}(\bold{x},t)&=&B_{1^n}(t)s_{1^n}+\sum_{\ell=1}^{k}B_{2^{\ell}1^{n-2\ell}}(t)s_{2^{\ell}1^{n-2\ell}}\\
 &=& B_{1^n}(t)e_n+\sum_{\ell=1}^{k}B_{2^{\ell}1^{n-2\ell}}(t)(e_{(n-\ell,\ell)}-e_{(n-\ell+1,\ell-1)})\\
 &=& \sum_{\ell=0}^{k-1}\{B_{2^\ell1^{n-2\ell}}(t)-B_{2^{\ell+1}1^{n-2\ell-2}}(t)\}e_{(n-\ell,\ell)}+B_{2^k1^{n-2k}}(t)e_{(n-k,k)}
\end{eqnarray*}
Note by Remark \ref{rem:uni}, if we show that $B_{2^{\ell}1^{n-2\ell}}(t)-B_{2^{\ell+1}1^{n-2\ell-2}}(t)$ is a polynomial in $t$ with nonnegative coefficients for each $\ell=0,1,\dots,k-1$, then the $e$-positivity of $X_{G}(\bold{x},t)$ is proved. We also know from Theorem \ref{thm:Gasha} that $B_{\la}(t)=\sum_{T\in\mathcal{T}_{P,\la}}t^{\rm{inv}_{G}(T)}$, where $\mathcal{T}_{P,\la}$ is the set of $P$-tableaux of shape $\la$. Our strategy to show the $e$-positivity of $X_{G}(\bold{x},t)$ is to find an weight(inv) preserving injection from $\mathcal{T}_{P,2^{\ell+1}1^{n-2\ell-2}}$ to $\mathcal{T}_{P,2^{\ell}1^{n-2\ell}}$. \\ 

We first define a subset of $\mathcal{T}_{P,2^{\ell}1^{n-2\ell}}$, that will be shown to be in (weight preserving) bijection with $\mathcal{T}_{P,2^{\ell+1}1^{n-2\ell-2}}$.

For a fixed $\ell\in\{0,1,\dots,k-1\}$, let $\mathcal{T'}_{P,2^{\ell}1^{n-2\ell}}$ be the subset of $\mathcal{T}_{P,2^{\ell}1^{n-2\ell}}$, each of whose elements $T=[a_{i,j}]$ has some $s \geq \ell+2$ satisfying $a_{i,1}<_Pa_{s,1}$ for all $i\in\{\ell+1, \ell+2, \dots, s-1\}$. If such $s$ exists, then it is unique since $a_{s,1}>r$ so that $a_{s,1}\not<_Pa_{j,1}$ for all $j>s$.

\vspace{1mm}
\begin{example}
Let $P=P(3,4,6,7,7,7)$ be a natural unit interval order. In this case, $n=7$ and $r=3$. Note that there are only eight relations in $P$,
$$1<_P4,~1<_P5,~1<_P6,~1<_P7,~2<_P5,~2<_P6,~2<_P7,~3<_P7.$$ 
Let us consider two $P$-tableaux $T_1=[a_{i,j}]$ and $T_2=[b_{i,j}]$ in $\mathcal{T}_{P,2^11^5}$ with $\ell=1$.\\

\begin{center}
\begin{ytableau}      
1&5\cr
2\cr
4\cr
6\cr
3\cr
7\cr
\end{ytableau}
\qquad\qquad\qquad\qquad\qquad\begin{ytableau} 
1&5\cr
3\cr
2\cr
7\cr
6\cr
4\cr
\end{ytableau}
\end{center}
$$T_1 \qquad\qquad\qquad\qquad\qquad\qquad\quad T_2$$

\noindent We can see that $T_1\not\in\mathcal{T'}_{P,2^11^5}$ since $2\not<_P4=a_{3,1}$, $4\not<_P6=a_{4,1}$, $2\not<_P 3=a_{5,1}$, and $4\not<_P7=a_{6,1}$, while $T_2 \in \mathcal{T'}_{P,2^11^5}$ with $s=4$ since $3<_P7=b_{4,1}$ and $2<_P7=b_{4,1}$.
\end{example}

We now define a map $\psi_{\ell}:\mathcal{T}_{P,2^{\ell+1}1^{n-2\ell-2}} \rightarrow \mathcal{T}_{P,2^{\ell}1^{n-2\ell}}$. Given $T=[b_{i,j}] \in \mathcal{T}_{P,2^{\ell+1}1^{n-2\ell-2}}$, we let $s$ be the smallest $i>\ell+1$ such that $b_{i,1}\not<_Pb_{\ell+1,2}$; otherwise $s=n-\ell$, then let $\psi_{\ell}(T)$ be the tableau obtained by placing(inserting) $b_{\ell+1,2}$ right below $b_{s-1,1}$. It is easy to see that $\psi_{\ell}(T)$ is a $P$-tableau in $\mathcal{T'}_{P,2^{\ell}1^{n-2\ell}}$. Moreover, we can show that $\psi_{\ell}$ is a weight preserving bijection between $\mathcal{T}_{P,2^{\ell+1}1^{n-2\ell-2}}$ and $\mathcal{T'}_{P,2^{\ell}1^{n-2\ell}}$:

\noindent Define a map $\phi_{\ell}:\mathcal{T'}_{P,2^{\ell}1^{n-2\ell}} \rightarrow \mathcal{T}_{P,2^{\ell+1}1^{n-2\ell-2}}$ as follows. For a given $T=[a_{i,j}]\in\mathcal{T'}_{P,2^{\ell}1^{n-2\ell}}$, we move $a_{s,1}$ to the right side of ${a_{\ell+1,1}}$ to obtain $\phi_\ell(T)$. Note that $a_{\ell+1,1}<_Pa_{s,1}\not<_Pa_{\ell,2}$ since $a_{\ell,1}<_Pa_{\ell,2}$ and $a_{\ell,2}>r$. Moreover, $a_{s-1,1}<r$ gives $a_{s+1,1}\not<_Pa_{s-1,1}$. Thus $\phi_{\ell}(T)$ is a $P$-tableau in $\mathcal{T}_{P,2^{\ell+1}1^{n-2\ell-2}}$ and $\phi_\ell$ is the inverse map of $\psi_\ell$. Further, since $a_{i,1}<_P a_{s,1}$ there is no edge between $a_{i,1}$ and $a_{s,1}$ for all $i\in \{\ell+1,\ell+2,\dots, s-1\}$. Therefore, two $P$-tableaux $T$ and $\phi_{\ell}(T)$ have the same $G$-inversions.\\
 
\begin{example}
Let $P=P(3,4,6,7,7,7)$ be a natural unit interval order. The following figure shows a correspondence between tableaux $T\in\mathcal{T}_{P,2^21^3}$ and $\psi_{1}(T) \in \mathcal{T'}_{P,2^11^5}$.\\

\begin{center}
\begin{ytableau} 
1&5\cr
3&7\cr
2\cr
6\cr
4\cr
\end{ytableau}
\quad\quad $\leftrightarrow$ \quad\quad\begin{ytableau}      
1&5\cr
3\cr
2\cr
7\cr
6\cr
4\cr
\end{ytableau}
\end{center}
$$T\quad\quad\quad\quad\quad\quad\quad\quad \psi_1(T)$$
\end{example}

Let $\widebar{\mathcal{T'}}_{P,2^{\ell}1^{n-2\ell}}$ be the set $\mathcal{T}_{P,2^{\ell}1^{n-2\ell}}-\mathcal{T'}_{P,2^{\ell}1^{n-2\ell}}$. Then
$$B_{2^{\ell}1^{n-2\ell}}(t)-B_{2^{\ell+1}1^{n-2\ell-2}}(t)=\sum_{T\in\widebar{\mathcal{T'}}_{P,2^{\ell}1^{n-2\ell}}}t^{{\rm inv}_{G}(T)},$$
and this gives a combinatorial interpretation of the coefficient of $e_{(n-\ell,\ell)}$ in $X_{G}(\bold{x},t)$ for a natural unit interval order $P=(r,m_2,m_3,\dots,m_r,n,\dots,n)$ of size $n$. \\

We can now state our main result of this section.

\vspace{1mm}
\begin{thm}\label{thm:main} Let $P=P(m_1,m_2,\dots, m_{n-1})$ be a natural unit interval order and let $G$ be the incomparability graph of $P$. If $m_1=r<n$ and $m_{r+1}=n$, then
 $$X_{G}(\bold{x},t)=\sum_{\ell=0}^{k}C_{(n-\ell,\ell)}(t)e_{(n-\ell,\ell)},$$ 
where $k$ is a positive integer and 
$$C_{(n-\ell,\ell)}(t)=\sum_{T\in\widebar{\mathcal{T'}}_{P,2^{\ell}1^{n-2\ell}}}t^{\rm{inv}_G(T)}.$$
Here, $\widebar{\mathcal{T'}}_{P,2^{\ell}1^{n-2\ell}}$ is the subset of $\mathcal{T}_{P,2^{\ell}1^{n-2\ell}}$ whose every element $T=[a_{i,j}]$ has no $s\geq\ell+2$ such that $a_{i,1}<_P a_{s,1}$ for all $i\in \{\ell+1, \ell+2, \dots, s-1\}$.\\
Consequently $X_{G}(\bold{x},t)$ is $e$-positive.
\end{thm}

\vspace{1mm}
\begin{example} Let $P=P(\bold{m})$ be a natural unit interval order with $\bold{m}=(2,3,4)$. In this case, $n=4$, $r=2$, $m_3=4$, and the incomparability graph $G$ is the path $1-2-3-4$. There are $14$ $P$-tableaux:\\[2ex]

$T_1$=\begin{ytableau}      
1&3\cr
2&4\cr
\end{ytableau} 
$~T_2$=\begin{ytableau}      
2&4\cr
1&3\cr
\end{ytableau} 
$~T_3$=\begin{ytableau}      
1&4\cr
2\cr
3\cr
\end{ytableau} 
$~T_4$=\begin{ytableau}      
1&4\cr
3\cr
2\cr
\end{ytableau} 
$~T_5$=\begin{ytableau}      
1&3\cr
2\cr
4\cr
\end{ytableau}
$~T_6$=\begin{ytableau}      
2&4\cr
1\cr
3\cr
\end{ytableau} \\[2ex]

$T_7$=\begin{ytableau}      
1\cr
2\cr
3\cr
4\cr
\end{ytableau}
$\quad T_8$=\begin{ytableau}      
2\cr
1\cr
3\cr
4\cr
\end{ytableau}
$\quad T_9$=\begin{ytableau}      
3\cr
2\cr
1\cr
4\cr
\end{ytableau}
$\quad T_{10}$=\begin{ytableau}      
4\cr
3\cr
2\cr
1\cr
\end{ytableau}
$\quad T_{11}$=\begin{ytableau}      
1\cr
2\cr
4\cr
3\cr
\end{ytableau}
$\quad T_{12}$=\begin{ytableau}      
1\cr
4\cr
3\cr
2\cr
\end{ytableau}
$\quad T_{13}$=\begin{ytableau}      
1\cr
3\cr
2\cr
4\cr
\end{ytableau}
$\quad T_{14}$=\begin{ytableau}      
2\cr
1\cr
4\cr
3\cr
\end{ytableau}\\[2ex]

By definition, we have 
$\widebar{\mathcal{T'}}_{P,2^{2}}=\{T_1,T_2\}$,  $\widebar{\mathcal{T'}}_{P,21^{2}}=\{T_3,T_4\}$, and $\widebar{\mathcal{T'}}_{P,1^{4}}=\{T_7,T_8,T_9,T_{10}\}$. Therefore, by Theorem \ref{thm:main},  

\begin{eqnarray*}
X_{G}(\bold{x},t)&=&\left( \sum_{i=1}^{2}t^{{\rm inv}_G(T_i)} \right)e_{(2,2)}+\left( \sum_{i=3}^{4}t^{{\rm inv}_G(T_i)} \right)e_{(3,1)}+\left( \sum_{i=7}^{10}t^{{\rm inv}_G(T_i)} \right)e_{4}\\
&=&(t+t^2)e_{(2,2)}+(t+t^2)e_{(3,1)}+(1+t+t^2+t^3)e_{4}\\
&=&t[2]_te_{(2,2)}+t[2]_te_{(3,1)}+[4]_t
e_4.
\end{eqnarray*}

\vspace{3mm}
More precisely, $\mathcal{T'}_{P,21^{2}}=\{T_5=\psi_{1}(T_1),T_6=\psi_{1}(T_2)\}$ and $\mathcal{T'}_{P,1^{4}}=\{T_{11}=\psi_{0}(T_3),T_{12}=\psi_{0}(T_4),T_{13}=\psi_{0}(T_5),T_{14}=\psi_{0}(T_6)\}$.
\end{example}

\vspace{3mm}
For a fixed $r>1$ and $n$, there are $\binom{n-1}{r-1}-1$ sequences of $(r,m_2, \dots, m_r, n, \dots,n)$ with $m_i\neq i$ for all $i=2,3\dots,r$. Therefore, Theorem \ref{thm:main} covers $2^{n-1}-n$ natural unit interval orders among $C_{n-1}$ natural unit interval orders on $n$ elements having connected incomparability graphs. Note that the natural unit interval orders  $P=P(m_1, m_2,n, \dots, n)$ considered by Shareshian-Wachs as stated in Remark \ref{rem:uni} (2) are $\frac{n^2-n-4}{2}$ special cases of the ones in Theorem \ref{thm:main}.\\


Next two corollaries follow easily from Theorem \ref{thm:main}.

\begin{cor} \label{cor:graph}
For $G={\rm inc}(P)$ of some natural unit interval order $P$ on $[n]$, if induced subgraphs on $\{1,2,\dots,r\}$ and $\{r+1,r+2, \dots,n\}$ are complete graphs for some $r$, then $X_{G}(\bold{x},t)$ is $e$-positive. 
\end{cor}

Corollary \ref{cor:graph} is a quasi-refinement of Stanley-Stembridge's result stated in Remark \ref{rem:uni} (5).\\

Recall the graph $G_{n,r}$ in Example \ref{ex:stair}. The graph $G_{n,1}$ is the trivial graph and the graph $G_{n,n}$ is the complete graph. Hence, $X_{G_{n,r}}$ is $e$-positive when $r=1$ and $r=n$. Shareshian and Wachs proved that $X_{G_{n,r}}(\bold{x},t)$ is $e$-positive for $r=2,n-2,n-1$. The following corollary implies that $G_{n,r}$ is $e$-positive for all $r\geq  \lceil \frac{n}{2} \rceil$.

\begin{cor}
For a natural unit interval order $P_{n,r}=P(\bold{m})$ with $n$ elements and $\bold{m}=(r,r+1,r+2,\dots,n, \dots,n)$, let $G_{n,r}$ be the incomparability graph of $P_{n,r}$. If $m_1=r \geq \lceil \frac{n}{2} \rceil$, then $m_{r+1}=n$ so that $X_{G_{n,r}}(\bold{x},t)$ is $e$-positive.
\end{cor}

\subsection{The second $e$-positive class} \label{sec:secondmain}

In this section, we consider the natural unit interval orders $P=P(r,n-1,m_3,\dots,m_{n-2},n)$ for some $r>1$. Note that there must be some $3\leq s \leq n-2$ such that $m_s=n-1$ and $m_{s+1}=n$. Let $G$ be the incomparability graph of $P$. Then the induced subgraph of $G$ on $\{2,3,\dots,n-1\}$ is a complete graph. If $s \leq r$, then $m_{r+1}=n$ and $X_{G}(\bold{x},t)$ is $e$-positive by Theorem \ref{thm:main}. Hence, we consider the case when $s>r$ and $n\geq 5$. The first figure in Figure \ref{Fig4} shows a corresponding Catalan path (solid line) of $P$. The dotted line indicates the bounce path and we can see that the bounce number is three for the unit interval order we are considering. The second figure in Figure \ref{Fig4} shows an example of $G={\rm inc}(P)$ with $P=(4,8,8,8,8,8,9,9)$.

\begin{figure}[h]
\centering {\includegraphics[height=4.3cm]{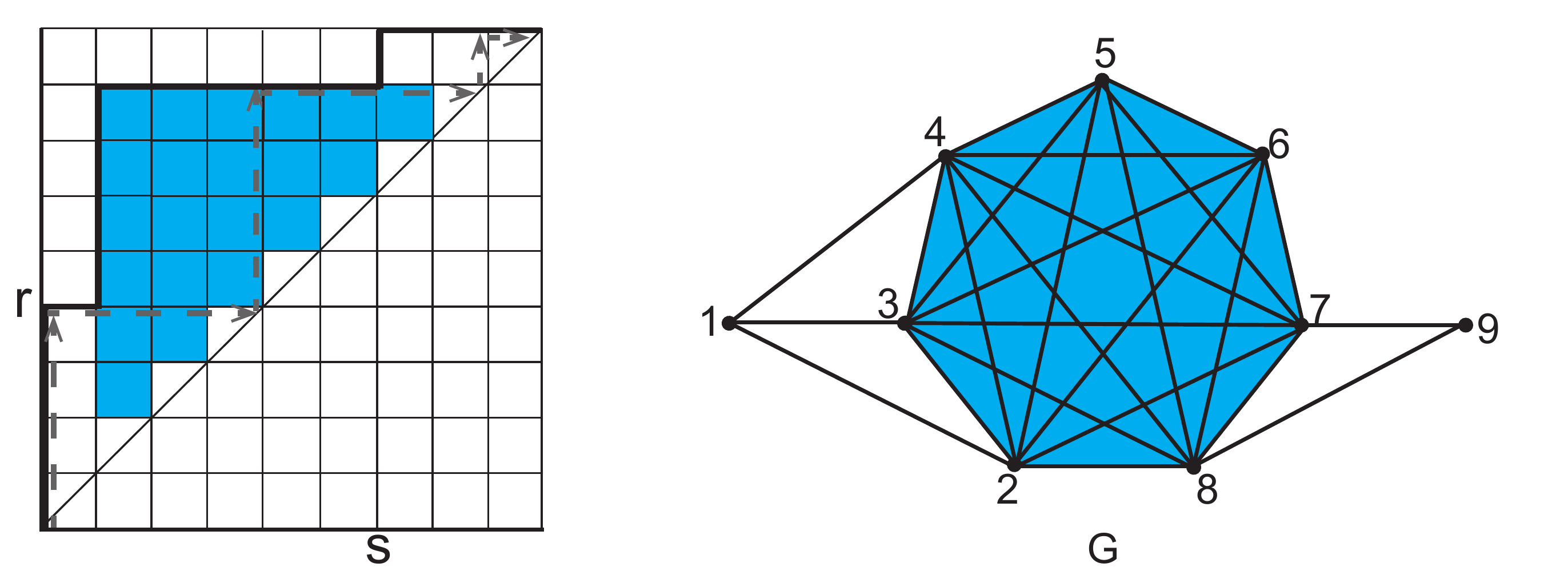} }
\caption{A Catalan path and the bounce path of $P(r,n-1, \dots,n-1, m_{r+2}, \dots, m_{n-2}, n)$}\label{Fig4}
\end{figure}

\noindent We know, by Lemma \ref{lem:chain}, there is no $P$-tableau of shape $\la$ if $\la_1>3$. Since the complete graph on $\{2,3,\dots,n-1\}$ is a subgraph of the incomparability graph of $P$, $a<_Pb$ means that $a=1$ or $b=n$. Therefore, if $T$ is a $P$-tableau, then $T$ must have one of four shapes, $3^11^{n-3}$, $2^21^{n-4}$, $2^11^{n-2}$, and $1^n$. It is also easy to see that there is no $P$-tableau of shape $2^21^{n-4}$ having $1$ and $n$ in the same row. By using a similar method we used in Section \ref{sec:firstmain}, we can prove the following theorem.

\begin{thm} \label{thm:second}
Let $P=P(r,m_2,m_3,\dots,m_{n-1})$ be a natural unit interval order with $m_{s}=n-1$ and $m_{s+1}$=n for some $r<s<n-1$, then
$$X_{G}(\bold{x},t)=\sum_{\la \in A}B_{\la}(t)s_{\la}=\sum_{\mu \in A'} C_{\mu}(t)e_{\mu},$$
where $A=\{1^n, 2^11^{n-2}, 2^21^{n-4}, 3^11^{n-3}\}$ and $A'=\{(n-2,1,1), (n-2,2), (n-1,1), n\}$. Moreover, the coefficients of the $e$-expansion
\begin{eqnarray*}
C_{(n-2,1,1)}(t)&=& B_{3^11^{n-3}}(t),\\
C_{(n-2,2)}(t)&=& B_{2^21^{n-4}}(t)-B_{3^11^{n-3}}(t),\\
C_{(n-1,1)}(t)&=&B_{2^11^{n-2}}(t)-B_{2^21^{n-4}}(t)-B_{3^11^{n-3}}(t), \mbox{ and}\\
C_{n}(t)&=&B_{1^n}(t)-B_{2^11^{n-2}}(t)+B_{3^11^{n-3}}(t)
\end{eqnarray*}
are polynomials in $t$ with nonnegative coefficients. Consequently $X_{G}(\bold{x},t)$ is $e$-positive.
\end{thm}

\begin{proof}
Since any $B_{\la}(t)$ is a polynomial with nonnegative coefficients, $C_{(n-2,1,1)}(t)$ also has only nonnegative coefficients.\\ 
\indent To show that $C_{n}(t)=B_{1^n}(t)-B_{2^11^{n-2}}(t)+B_{3^11^{n-3}}(t)$ is a nonnegative polynomial, we let $E=\{~T=[a_{i,j}]\in \mathcal{T}_{P,2^11^{n-2}}~|~a_{1,1}<_P a_{2,1}<_P a_{1,2}~\}$ be a subset of $\mathcal{T}_{P,2^11^{n-2}}$ and construct two weight(inversion) preserving injections; one is from $E$ to $\mathcal{T}_{P,3^11^{n-3}}$ and the other is from $\mathcal{T}_{P,2^11^{n-2}}-E$ to $\mathcal{T}_{P,1^n}$. Note that if $a_{1,1}<_Pa_{2,1}<_Pa_{1,2}$ in a $P$-tableau $T=[a_{i,j}]\in \mathcal{T}_{P,2^11^{n-2}}$, then $a_{1,1}=1$ and $a_{1,2}=n$. We hence define an injection from $E$ to $\mathcal{T}_{P,3^11^{n-3}}$ as it is depicted in the following figure;

\begin{center}
\begin{ytableau}      
1&n\cr
a_{2,1}\cr
\vdots\cr
\cr
\cr
\cr
\end{ytableau}
\quad\quad $\mapsto$ \quad\quad\begin{ytableau} 
1&a_{2,1}&n\cr
\vdots\cr
\cr
\cr
\cr
\end{ytableau}
\end{center}

\noindent Now we define an injection from $\mathcal{T}_{P,2^11^{n-2}}-E$ to $\mathcal{T}_{P,1^n}$. For an element $T=[a_{i,j}]$  of $\mathcal{T}_{P,2^11^{n-2}}-E$, we let $s$ be the smallest $i\geq2$ such that $a_{i,1}\not<_Pa_{1,2}$; otherwise $s=n$, and then we place $a_{1,2}$ right below $a_{s-1,1}$ to obtain a $P$-tableau of shape $1^n$:

\begin{center}
\begin{ytableau}      
a_{1,1}&a_{1,2}\cr
\vdots\cr
a_{s,1}\cr
\vdots\cr
\cr
\cr
\end{ytableau}
\quad\quad $\mapsto$ \quad\quad\begin{ytableau} 
a_{1,1}\cr
\vdots\cr
a_{1,2}\cr
a_{s,1}\cr
\vdots\cr
\cr
\cr
\end{ytableau}
\end{center}
We can easily check that the resulting tableaux are all distinct and inversions are preserved in both injections.

Similarly, to show that $C_{(n-1,1)}(t)=B_{2^11^{n-2}}(t)-B_{2^21^{n-4}}(t)-B_{3^11^{n-3}}(t)$ is a nonnegative polynomial, we let $F=\{~T=[b_{i,j}]\in \mathcal{T}_{P,2^11^{n-2}}~|~b_{1,1}=1 \text{~and~} b_{1,2}=n~\}$ be a subset of $\mathcal{T}_{P,2^11^{n-2}}$ and construct two weight(inversion) preserving injections; one is from $\mathcal{T}_{P,3^11^{n-3}}$ to $F$ and the other is from  $\mathcal{T}_{P,2^21^{n-4}}$ to $ \mathcal{T}_{P,2^11^{n-2}}-F$. We define an injection from $\mathcal{T}_{P,3^11^{n-3}}$ to $F$ as it is depicted in the following figure;

\begin{center}
\begin{ytableau}      
1&a_{1,2}&n\cr
\vdots\cr
\cr
\cr
\cr
\end{ytableau}
\quad\quad $\mapsto$ \quad\quad\begin{ytableau} 
1&n\cr
a_{1,2}\cr
\vdots\cr
\cr
\cr
\cr
\end{ytableau}
\end{center}
Now we define an injection from $\mathcal{T}_{P,2^21^{n-4}}$ to $ \mathcal{T}_{P,2^11^{n-2}}-F$. For an element $T=[a_{i,j}]$  of $\mathcal{T}_{P,2^21^{n-4}}$, let $s$ be the smallest $i\geq3$ such that $a_{i,1}\not<a_{2,2}$; otherwise $s=n-1$, and then we place $a_{2,2}$ right below $a_{s-1,1}$ to obtain a $P$-tableau of shape $2^11^{n-2}$:

\begin{center}
\begin{ytableau}      
a_{1,1}&a_{1,2}\cr
a_{2,1}&a_{2,2}\cr
\vdots\cr
a_{s,1}\cr
\vdots\cr
\end{ytableau}
\quad\quad $\mapsto$ \quad\quad\begin{ytableau} 
a_{1,1}&a_{1,2}\cr
a_{2,1}\cr
\vdots\cr
a_{2,2}\cr
a_{s,1}\cr
\vdots\cr
\end{ytableau}
\end{center}
Since there is no $P$-tableau of shape $2^21^{n-4}$ having $1$ and $n$ in the same row,
$(a,b)\neq (1,n)$, so the resultant $P$-tableau is in the set $\mathcal{T}_{P,2^11^{n-2}}-F$. In both injections, the resulting tableaux are all distinct and inversions are preserved.\\

Lastly, for $C_{(n-2,2)}(t)$, we define an injection from $\mathcal{T}_{P,3^11^{n-3}}$ to $\mathcal{T}_{P,2^21^{n-4}}$. Let $T=[a_{i,j}]$ be an element of the set $\mathcal{T}_{P,3^11^{n-3}}$. If $a_{2,1}<_Pn$, then the injection is described as in the following left figure. On the other hand, if $a_{2,1}\not<_Pn$, then the injection is described as the following right figure:

\begin{center}
\begin{ytableau}
1&a_{1,2}&n\cr
a_{2,1}\cr
\vdots\cr
\cr
\cr
\end{ytableau}
\quad\quad $\mapsto$ \quad\quad\begin{ytableau} 
1&a_{1,2}\cr
a_{2,1}&n\cr
\vdots\cr
\cr
\cr
\end{ytableau}
\quad\quad and \quad\quad
\begin{ytableau}
1&a_{1,2}&n\cr
a_{2,1}\cr
\vdots\cr
\cr
\cr
\end{ytableau}
\quad\quad $\mapsto$ \quad\quad\begin{ytableau} 
1&a_{2,1}\cr
a_{1,2}&n\cr
\vdots\cr
\cr
\cr
\end{ytableau}
\end{center} 
Note that $1<_Pa$ and $b<_Pn$ for only $a>r$ and $b\leq s$. Hence, as in the second injection, if  $1<_Pa_{1,2}<_P n$ and $a_{2,1}\not<_Pn$, then $r<a_{1,2}\leq s<a_{2,1}$, and then $1<_Pa_{2,1}$. Since $n$ is the largest element in $P$, $a_{2,1}\not>_P n$. Therefore, the resultant of the second injection is a $P$-tableau. Moreover, the second injection preserves the weight since $a_{1,2}\not<_Pa_{2,1}$ and $a_{1,2}<a_{2,1}$.  
\end{proof}

\vspace{3mm}
For a fixed $n$, Theorem \ref{thm:second} covers $(n-3)(n-4)/2$ more natural unit interval orders on $n$ elements having connected incomparability graphs.

\begin{cor}
Let $G$ be the incomparability graph of a natural unit interval order on $[n]$. If $G$ has the complete graph on $\{2,3, \dots,n-1\}$ as a subgraph, then $X_{G}(\bold{x},t)$ is $e$-positive. 
\end{cor}

\vspace{3mm}
\section{Some explicit formulae} \label{sec:last}

In this section we prove the $e$-unimodality conjecture of $X_G({\bf x}, t)$ for two subclasses of the class of natural unit interval orders that we considered in Section \ref{sec:firstmain}. We use the combinatorial model we obtained in Theorem \ref{thm:main} to find explicit formulae of the coefficients $C_{(n-\ell, \ell)}(t)$'s, which shows that they are unimodal with center of symmetry $\frac{|E(G)|}{2}$. By Remark \ref{rem:uni}, this show the $e$-unimodality of $X_{G}(\bold{x},t)$.

\subsection{The first explicit formula} \label{sec:firstform}

In this subsection we consider natural unit interval orders $P=P(m_1,m_2 \dots, m_{n-1})$ with $m_1=\cdots=m_{s}=r$ and $m_{s+1}=\cdots=m_{n-1}=n$ for some positive integers $s\leq r\leq n-1$. Figure \ref{Fig5} shows a corresponding Catalan path of $P$. They form a subclass of the natural unit interval orders we considered in Section~\ref{sec:firstmain} and we know that the chromatic quasisymmetric functions of them are $e$-positive.

\begin{figure}[h]
\centering {\includegraphics[height=4.3cm]{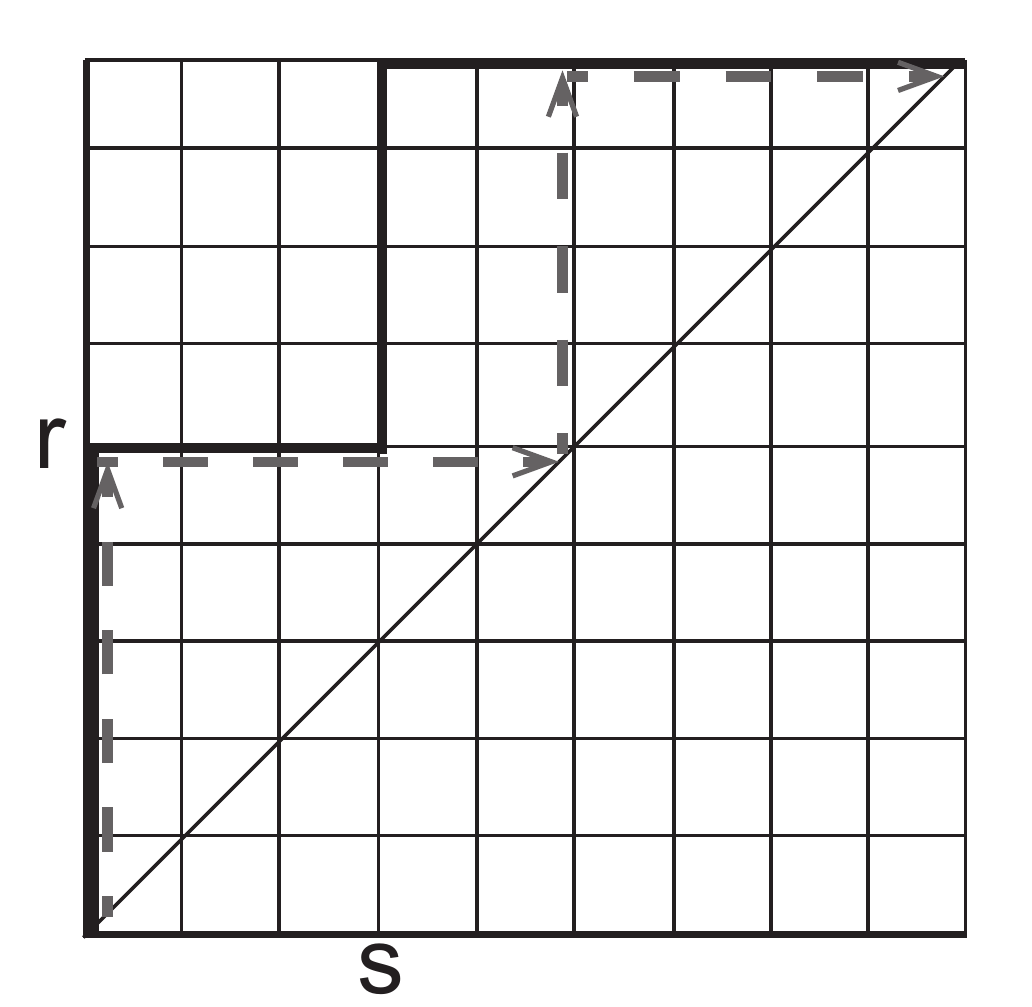} }
\caption{A Catalan path and the bounce path of $P(r, \dots, r, n, \dots, n)$}\label{Fig5}
\end{figure}

We first state a well-known lemma that is useful for our work.

\begin{lem}[Stanley \cite{S2}]\label{lem:Stan} 
Let $M=\{1^{c_1}, 2^{c_2}, \dots, n^{c_n}\}$ be a multiset of cardinality $n=c_1+c_2+\cdots+c_n$ and let $S(M)$ be the set of permutations on $M$. An {\rm inversion} of a permutation $\pi=\pi_1\pi_2\cdots \pi_n\in S(M)$ is a pair $(\pi_i, \pi_j)$ with $i<j$ and $\pi_i>\pi_j$. Then 
$$\sum_{\pi \in S(M)} t^{i(\pi)}=\frac{[n]_t!}{[c_1]_t! \cdots [c_n]_t!},$$
where $i(\pi)$ is the number of inversions of $\pi$.
\end{lem}

By using Theorem \ref{thm:en}, Theorem \ref{thm:main}, and Lemma \ref{lem:Stan}, we have an explicit formula for $X_{G}(\bold{x},t)$ in the $e$-basis expansion.

\begin{thm} \label{thm:form1}
Let $P=P(m_1,m_2 \dots, m_{n-1})$ be a natural unit interval order and $r,~s$ be positive integers satisfying $s\leq r \leq n-1$. Let $G$ be the incomparability graph of $P$. If $m_1=\cdots=m_s=r$ and $m_{s+1}=n$, then 
$$X_{G}(\bold{x},t)=\sum_{\ell=0}^{\min\{n-r,s\}}t^{(r-s)\ell} \frac{[n-r]_t! [s]_t! [r-\ell-1]_t! [n-s-\ell-1]_t! [n-2\ell]_t}{[n-r-\ell]_t! [s-\ell]_t! [r-s-1]_t!} e_{(n-\ell,\ell)}.$$
Consequently, $X_{G}(\bold{x},t)$ is $e$-positive and $e$-unimodal with center of symmetry $\frac{\binom{n}{2}-(n-r)s}{2}$.
\end{thm} 

\begin{proof}
If we let $A_1=\{1,\dots, s \}$ and $A_2=\{r+1, \dots, n\}$, then any two elements in $A_i$ are incomparable for each $i=1,2$, and $a<_Pb$ for any $a\in A_1$ and $b\in A_2$. Note that if $T=[a_{i,j}]\in\widebar{\mathcal{T'}}_{P,2^{\ell}1^{n-2\ell}}$, then $\{a_{i,1}~|~1\leq i \leq \ell\}\subset A_1$ and $\{a_{i,2}~|~1\leq i \leq \ell\}\subset A_2$.\\ 
For any $B_1\subset A_1$ and $B_2 \subset A_2$ with $|B_1|=|B_2|=\ell$, let $\widebar{\mathcal{T'}}_P(B_1,B_2)$ be the set of $T=[a_{i,j}]$ in $\widebar{\mathcal{T'}}_{P,2^{\ell}1^{n-2\ell}}$ satisfying 
$\{a_{i,1}~|~1\leq i \leq \ell \}=B_1$ and $\{a_{i,2}~|~1\leq i \leq \ell \}=B_2$.
Then 
$$\widebar{\mathcal{T'}}_{P,2^{\ell}1^{n-2\ell}}=\bigcup_{\substack{B_1\subset A_1\\ B_2\subset A_2}}\widebar{\mathcal{T'}}_P(B_1,B_2).$$

By Theorem \ref{thm:main}, we know that the coefficient $C_{(n-\ell,\ell)}(t)$ of $e_{(n-\ell,\ell)}$ in the $e$-expansion of $X_{G}(\bold{x},t)$ is equal to the sum of $t^{{\rm inv}_{G}(T)}$'s for all $T\in\widebar{\mathcal{T'}}_{P,2^{\ell}1^{n-2\ell}}$. Hence, 
$$C_{(n-\ell,\ell)}(t)=\sum_{\substack{B_1\subset A_1\\B_2\subset A_2}}\sum_{T\in \widebar{\mathcal{T'}}_P(B_1,B_2)}t^{{\rm inv}_G(T)}.$$

For convenience, we divide the set of $G$-inversions of $T \in \widebar{\mathcal{T'}}_P(B_1,B_2)$ into three sets. We let ${\rm row}(i)$ denote the row in which $i$ is located in $T$: \\

\begin{enumerate} 
\item[$\bullet$] $I_1(T)=\{\{i,j\}\in E(G)~|~i<j, ~{\rm row}(j)<{\rm row}(i)\leq \ell \}$.
\item[$\bullet$] $I_2(T)=\{\{i,j\}\in E(G)~|~i<j, ~\ell<{\rm row}(j)<{\rm row}(i) \}$.
\item[$\bullet$] $I_3(T)=\{\{i,j\}\in E(G)~|~i<j, ~{\rm row}(j)\leq \ell<{\rm row}(i) \}$.
\end{enumerate}

\vspace{3mm}
\noindent Then ${\rm inv}_{G}(T)=|I_1(T)|+|I_2(T)|+|I_3(T)|$. Note that $I_3(T_1)=I_3(T_2)$ for any $T_1, T_2\in \widebar{\mathcal{T'}}_P(B_1,B_2)$. Hence we can use $I_3(B_1,B_2)$ to denote the set $I_3(T)$  for any $T \in \widebar{\mathcal{T'}}_P(B_1,B_2)$, and we have
$$C_{(n-\ell,\ell)}(t)=\sum_{\substack{B_1\subset A_1\\B_2\subset A_2}}\left(t^{|I_3(B_1,B_2)|}\sum_{T\in \widebar{\mathcal{T'}}_P(B_1,B_2)}t^{|I_1(T)|+|I_2(T)|}\right).$$
For any $B_1\subset A_1$ and $B_2\subset A_2$ with $|B_1|=|B_2|=\ell$, let $Q(B_1,B_2)$ and $R(B_1,B_2)$ be the subposets of $P$ with elements $B_1\cup B_2$ and $[n]-(B_1\cup B_2)$, respectively. Note that the subposet $Q(B_1,B_2)$ is equivalent to the unit interval order $Q_{\ell}=P(m_1,m_2,\dots,m_{2\ell-1})$ with $m_1=m_2=\dots=m_{\ell}=\ell$ and $m_{\ell+1}=2\ell$ and the subposet $R(B_1,B_2)$ is equivalent to the unit interval order $R_{\ell}=P(m_1,m_2,\dots, m_{n-2\ell-1})$ with $m_1=m_2=\dots =m_{s-\ell}=r-\ell$ and $m_{s-\ell+1}=n$. Therefore, for any fixed $B_1\subset A_1$ and $B_2\subset A_2$ with $|B_1|=|B_2|=\ell$,
$$\sum_{T\in \widebar{\mathcal{T'}}_P(B_1,B_2)}t^{|I_1(T)|+|I_2(T)|}=
\left(\sum_{ T^u\in \widebar{\mathcal{T'}}_{Q_{\ell},2^{\ell} }}t^{{\rm inv}_{{\rm inc}(Q_{\ell})}(T^u)}\right) \left(\sum_{ T^b\in \widebar{\mathcal{T'}}_{R_{\ell},1^{n-2\ell} }}t^{{\rm inv}_{{\rm inc}(R_{\ell})}(T^b)}\right).$$ 
If $T^u\in \widebar{\mathcal{T'}}_{Q_{\ell},2^{\ell}}$, then the first column of $T^u$ must be occupied by the elements of $B_1$ and the second column of $T^u$ must be occupied by the elements of $B_2$. Let $S(B_i)$ be the set of permutations on $B_i$ for $i=1,2$, then 
$$\sum_{ T^u\in \widebar{\mathcal{T'}}_{Q_{\ell},2^{\ell} }}t^{{\rm inv}_{{\rm inc}(Q_{\ell})}(T^u)}=\left(\sum_{\pi_1 \in S(B_1)}t^{i(\pi_1)}\right) \left(\sum_{\pi_2 \in S(B_2)}t^{i(\pi_2)}\right),$$ 
where $i(\pi_1)$ and $i(\pi_2)$ are the numbers of inversions of permutations. Note that there is no inversion $\{i,j\}$ for $i\in B_1$, $j\in B_2$ since $i<_P j$ for any  $i\in B_1$, $j\in B_2$. By using Lemma \ref{lem:Stan}, we can evaluate each sums and we have $\sum_{\pi_1 \in S(B_1)}t^{i(\pi_1)}=\sum_{\pi_2 \in S(B_2)}t^{i(\pi_2)}=[\ell]_t!.$\\
\noindent On the other hand, since $R_{\ell}=P(m_1,m_2,\dots, m_{n-2\ell-1})$ with $m_1=m_2=\dots =m_{s-\ell}=r-\ell$ and $m_{s-\ell+1}=n$, by using Theorem \ref{thm:en}, we have
$$\sum_{ T^b\in \widebar{\mathcal{T'}}_{R_{\ell},1^{n-2\ell} }}t^{{\rm inv}_{{\rm inc}(R_{\ell})}(T^b)}=[n-2\ell]_t [r-\ell-1]_t! \frac{ [n-s-\ell-1]_t!}{[r-s-1]_t!}.$$
Therefore, 
$$\sum_{T\in \widebar{\mathcal{T'}}_P(B_1,B_2)}t^{|I_1(T)|+|I_2(T)|}=\frac{([\ell]_t!)^2[n-2\ell]_t [r-\ell-1]_t! [n-s-\ell-1]_t!}{[r-s-1]_t!},$$
and then
$$C_{(n-\ell,\ell)}(t)=\frac{([\ell]_t!)^2[n-2\ell]_t [r-\ell-1]_t! [n-s-\ell-1]_t!}{[r-s-1]_t!}\sum_{\substack{B_1\subset A_1\\B_2\subset A_2}}t^{|I_3(B_1,B_2)|}.$$

Now, we consider the sum $\sum t^{|I_3(B_1,B_2)|}$. For given $B_1=\{b_{11},b_{12},\dots,b_{1\ell}\}\subset A_1$ and $B_2=\{b_{21},b_{22},\dots,b_{2\ell}\} \subset A_2$, we divide $I_3(B_1,B_2)$ into three sets:\\

\begin{enumerate} 
\item[$\bullet$] $I_{31}(B_1,B_2)=\{\{i,j\}\in I_3(B_1,B_2)~|~i<j\leq s\}$,
\item[$\bullet$] $I_{32}(B_1,B_2)=\{\{i,j\} \in I_3(B_1,B_2)~|~r+1 \leq i<j \}$,
\item[$\bullet$] $I_{33}(B_1,B_2)=\{\{i,j\} \in I_3(B_1,B_2)~|~s+1\leq i \leq r <j \}$,
\end{enumerate}

\vspace{3mm}
\noindent so that $|I_3(B_1,B_2)|=|I_{31}(B_1,B_2)|+|I_{32}(B_1,B_2)|+|I_{33}(B_1,B_2)|$.\\
Note that $|I_{31}(B_1,B_2)|=\sum_{i=1}^{\ell}|\{a\in A_1~|~a<b_{1i}\}|$. Therefore, if we define a permutation $\sigma(B_1)=\sigma(B_1)_1 \sigma(B_1)_2 \dots \sigma(B_1)_{\ell}$ as $\sigma(B_1)_i=2$ when $i\in B_1$ and $\sigma(B_1)_i=1$ otherwise, then $\sigma(B_1)$ is a permutation on $\{1^{s-\ell}2^{\ell}\}$ with inversion $i(\sigma(B_1))=|I_{31}(B_1,B_2)|$.\\ 
Likewise, $|I_{32}(B_1,B_2)|=\sum_{i=1}^{\ell}|\{a\in A_2~|~a<b_{2i}\}|$. Hence, if we define a permutation $\sigma(B_2)=\sigma(B_2)_1 \sigma(B_2)_2 \dots \sigma(B_2)_{\ell}$ as $\sigma(B_2)_i=2$ when $i\in B_2$ and $\sigma(B_2)_i=1$ otherwise, then $\sigma(B_2)$ is a permutation on $\{1^{n-r-\ell}2^{\ell}\}$ with inversion $i(\sigma(B_2))=|I_{32}(B_1,B_2)|$.\\
Lastly, $|I_{33}(B_1,B_2)|$ is always $(r-s)\ell$. Therefore,
\begin{eqnarray*}
\sum_{\substack{B_1\subset A_1\\B_2\subset A_2}}t^{|I_3(B_1,B_2)|}&=&t^{(r-s)\ell}\sum_{\substack{B_1\subset A_1\\B_2\subset A_2}}t^{|I_{31}(B_1,B_2)|+|I_{32}(B_1,B_2)|}\\
&=&t^{(r-s)\ell}\sum_{\substack{\sigma_1\in S(M_1)\\ \sigma_2\in S(M_2) }}t^{i(\sigma_1)+i(\sigma_2)}\\
&=&t^{(r-s)\ell} \left(\sum_{\sigma_1\in S(M_1)}t^{i(\sigma_1)} \right) \left(\sum_{\sigma_2\in S(M_2)} t^{i(\sigma_2)} \right),
\end{eqnarray*}
where $S(M_i)$ is the set of permutations on a multiset $M_i$ with $M_1=\{1^{s-\ell}2^{\ell}\}$ and $M_2=\{1^{n-r-\ell}2^{\ell}\}$. By Lemma \ref{lem:Stan}, 
we have
$$\sum_{\sigma_1\in S(M_1)}t^{i(\sigma_1)}=\frac{[s]_t!}{[s-\ell]_t![\ell]_t!},$$
and
$$\sum_{\sigma_2\in S(M_2)}t^{i(\sigma_2)}=\frac{[n-r]_t!}{[n-r-\ell]_t![\ell]_t!}.$$

By combining all of these, we have

\begin{eqnarray*}
C_{(n-\ell,\ell)}(t)&=&\frac{([\ell]_t!)^2[n-2\ell]_t [r-\ell-1]_t! [n-s-\ell-1]_t!}{[r-s-1]_t!}\sum_{\substack{B_1\subset A_1\\B_2\subset A_2}}t^{|I_3(B_1,B_2)|}\\
&=& \frac{([\ell]_t!)^2[n-2\ell]_t [r-\ell-1]_t! [n-s-\ell-1]_t!}{[r-s-1]_t!} \frac{t^{(r-s)\ell} [s]_t! [n-r]_t!}{[s-\ell]_t![\ell]_t![n-r-\ell]_t![\ell]_t!} \\
&=&t^{(r-s)\ell} \frac{[n-r]_t! [s]_t! [r-\ell-1]_t! [n-s-\ell-1]_t! [n-2\ell]_t}{[n-r-\ell]_t! [s-\ell]_t! [r-s-1]_t!}
\end{eqnarray*}

as we desired.

Moreover, $t$-polynomials $\frac{[n-r]_t!}{[n-r-\ell]_t!}$, $\frac{[s]_t!}{[s-\ell]_t!}$, $\frac{[r-\ell-1]_t!}{[r-s-1]_t!}$, and $t^{(r-s)\ell}[n-s-\ell-1]_t! [n-2\ell]_t$ are unimodal with center of symmetries $\ell(2n-2r-\ell-1)/4$, $\ell(2s-\ell-1)/4$, $(s-\ell)(2r-s-\ell-3)/4$, and $\{2\ell(r-s)+\binom{n-s-\ell-1}{2}+n-2\ell-1\}/2$, respectively. By lemma \ref{lem:basic}, $C_{(n-\ell,\ell)}(t)$ is unimodal with center of symmetry
 $$\frac{\ell(2n-2r+2s-2\ell-2)+(s-\ell)(2r-s-\ell-3)+4\ell(r-s)+2\binom{n-s-\ell-1}{2}+2(n-2\ell-1)}{4}.$$
With some calculation, we can see that $C_{(n-\ell,\ell)}(t)$ is unimodal with center of symmetry  $\frac{\binom{n}{2}-(n-r)s}{2}$.
\end{proof}

The next two corollaries follow easily from Theorem \ref{thm:form1}.

\begin{cor} \label{cor:oneline}
Let $P=P(m_1,m_2 \dots, m_{n-1})$ be a natural unit interval order and $G$ be the incomparability graph of $P$. If $m_1=r$ for $r\leq n-1$ and $m_{2}=n$, then 
$$X_{G}(\bold{x},t)=[n-2]_t!([n]_t [r-1]_t e_n + t^{r-1}[n-r]_t e_{(n-1,1)}).$$
\end{cor}

\begin{cor}
Let $P=P(m_1,m_2 \dots, m_{n-1})$ be a natural unit interval order and $G$ be the incomparability graph of $P$. If $m_1=\cdots=m_{r-1}=r$ and $m_{r}=n$ for $r\leq n-1$ , then 
$$X_{G}(\bold{x},t)=\sum_{\ell=0}^{\min\{n-r,r-1\}}t^{\ell}[n-r]_t! [r-1]_t! [n-2\ell]_t e_{(n-\ell,\ell)}.$$
Here, the incomparability graph $G$ is $K_{(r,n-r+1)}$-chain mentioned in Remark~\ref{rem:uni} (3). 
\end{cor}

\subsection{The second explicit formula} \label{sec:secondform}

In this subsection we consider natural unit interval orders $P=P(m_1,m_2 \dots, m_{n-1})$ with $m_1=r$, $m_2=\cdots=m_{s}=n-1$, and $m_{s+1}=\cdots=m_{n-1}=n$ for some positive integers $s\leq r\leq n-2$, which form a subclass of the natural unit interval orders we considered in Section~\ref{sec:firstmain}. We know from Theorem~\ref{thm:main} that $X_{inc(P)}(\bold{x},t)$'s are $e$-positive, and we prove the $e$-unimodality of them.
Figure \ref{Fig6} shows a corresponding Catalan path of $P$ with the bounce path in dotted line.

\begin{figure}[h]
\centering {\includegraphics[height=4.3cm]{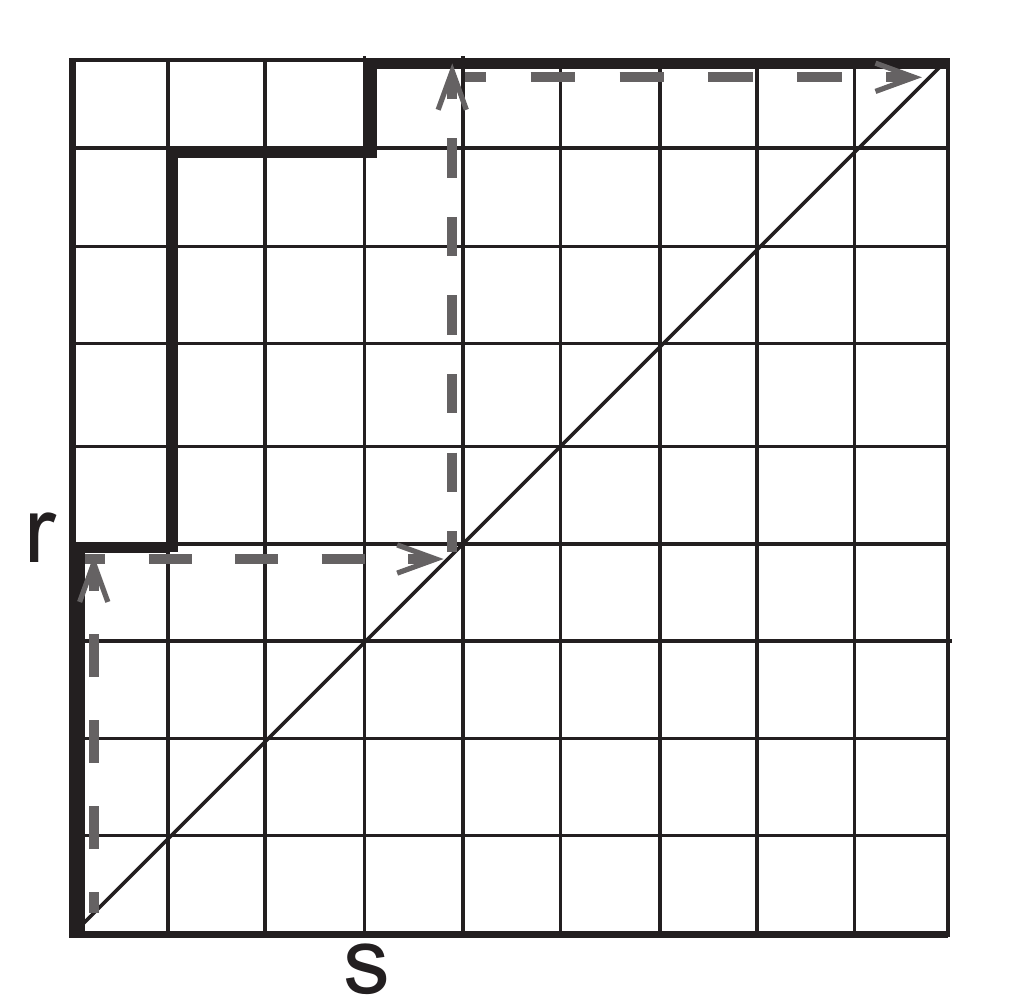} }
\caption{A Catalan path and the bounce path of $P(r, n-1,\dots, n-1, n, \dots, n)$}\label{Fig6}
\end{figure}

\begin{thm} \label{thm:form2}
Let $P=P(m_1,m_2 \dots, m_{n-1})$ be a natural unit interval order and $G$ be the incomparability graph of $P$. If $m_1=r$, $m_2=\cdots=m_s=n-1$, and $m_{s+1}=n$ for $2 \leq s\leq r \leq n-2$, then 
$$X_{G}(\bold{x},t)=[n-4]_t!\{\widebar{C}_{n}(t) e_{n}+\widebar{C}_{(n-1,1)}(t) e_{(n-1,1)}+\widebar{C}_{(n-2,2)}(t) e_{(n-2,2)}\},$$
where
\begin{eqnarray*}
\widebar{C}_n(t)&=&[n]_t [n-3]_t [r-1]_t [n-s-1]_t\\
\widebar{C}_{(n-2,2)}(t)&=&t^{n+r-s-3} [2]_t [n-r-1]_t [s-1]_t\\
\widebar{C}_{(n-1,1)}(t)&=& \frac{1}{2} [n-2]_t \{t^{r-1}(1+t^{n-r-s})[n-3]_t+t^{n-s-1} [2]_t [r-2]_t [s-1]_t\\
&~&\qquad\qquad\qquad\qquad\qquad\qquad\qquad\qquad +t^{r-1} [2]_t [n-r-1]_t [n-s-2]_t \} .
\end{eqnarray*}
Moreover, $X_{G}(\bold{x},t)$ is $e$-positive and $e$-unimodal with center of the symmetry $\frac{\binom{n}{2}-(n-r+s-1)}{2}$.
\end{thm}

\begin{proof}
By Theorem \ref{thm:main}, the coefficient $C_{(n-\ell,\ell)}(t)$ of $e$-basis expansion of $X_{G}(\bold{x},t)$ is equal to the sum of $t^{{\rm inv}_{G}(T)}$'s for all $T\in\widebar{\mathcal{T'}}_{P,2^{\ell}1^{n-2\ell}}$. Since $i=1$ or $j=n$ for all $i<_Pj$, $C_{(n-\ell,\ell)}(t)=0$ when $\ell>2$. By Theorem \ref{thm:en}, we have 
$$C_{n}(t)=[n]_t\prod_{i=2}^{n}[b_i]_t=[n]_t [n-3]_t! [r-1]_t [n-s-1]_t=[n-4]_t! \widebar{C}_n(t).$$ 
We already know from Theorem \ref{thm:en} that $C_{n}(t)$ is $e$-unimodal with center of symmetry $\frac{|E(G)|}{2}=\frac{\binom{n}{2}-(n-r+s-1)}{2}$.

Now we consider the sum of $t^{{\rm inv}_{G}(T)}$'s for all $T\in\widebar{\mathcal{T'}}_{P,2^{1}1^{n-2}}\cup \widebar{\mathcal{T'}}_{P,2^{2}1^{n-4}}$. Let $\widebar{\mathcal{T'}}_{P,(i,j)}$ be the set of all $P$-tableaux of $\widebar{\mathcal{T'}}_{P,2^{1}1^{n-2}}\cup \widebar{\mathcal{T'}}_{P,2^{2}1^{n-4}}$ such that the first row is occupied by $i$ and $j$. If $G-\{i,j\}$ is the induced subgraph of $G$ with vertex set $V(G)-\{i,j\}$, and 
$$X_{G-\{i,j\}}(\bold{x},t)=D_{n-2}(t)e_{n-2}+D_{(n-3,1)}(t)e_{(n-3,1)}, $$
then  
$$\sum_{T\in \widebar{\mathcal{T'}}_{P,(i,j)}} t^{{\rm inv}_{G}(T)}=t^{b_i+b_j}\{D_{n-2}(t)e_{(n-1,1)}+D_{(n-3,1)}(t)e_{(n-2,2)}\},$$
where $b_k=|\{ \{j,k\}\in E(G)~|~j<k\}|$.
We have three cases: 
\begin{enumerate}
\item[1.] Let $(i,j)=(1,n)$. Then $b_i+b_j=0+(n-s-1)$ and the subgraph $G-\{i,j\}$ is the complete graph with $n-2$ vertices. Since, $X_{G-\{i,j\}}(\bold{x},t)=[n-2]_t!e_{n-2}$,
$$\sum_{T\in \widebar{\mathcal{T'}}_{P,(1,n)}} t^{{\rm inv}_{G}(T)}=t^{n-s-1}[n-2]_t!e_{(n-1,1)}.$$
\item[2.] Let $(i,j)=(i,n)$ for $2\leq i \leq s$. Then $b_i+b_j=(i-1)+(n-s-1)$ and the subgraph $G-\{i,j\}$ is equal to the incomparability graph of a natural unit interval order $P'=(m_1', m_2', \dots, m_{n-3}')$ with $m_1'=r-1$ and $m_2'=n-2$. By Corollary \ref{cor:oneline}, 
$$X_{G-\{i,j\}}(\bold{x},t)=[n-4]_t!([n-2]_t[r-2]_t e_{n-2}+t^{r-2}[n-r-1]_t e_{(n-3,1)}).$$
It gives us \\
$\displaystyle{\sum_{i=2}^{s} \sum_{T\in \widebar{\mathcal{T'}}_{P,(i,n)}} t^{{\rm inv}_{G}(T)}=t^{n-s}[s-1]_t [n-4]_t!([n-2]_t[r-2]_t e_{(n-1,1)}}$
$$\qquad\qquad\qquad\qquad\qquad\qquad\qquad\qquad\qquad\qquad\qquad\qquad\qquad+t^{r-2}[n-r-1]_t e_{(n-2,2)}).$$
\item[3.] Let $(i,j)=(1,j)$ for $r+1\leq j \leq n-1$. Then $b_i+b_j=0+(j-2)$ and the subgraph $G-\{i,j\}$ is equal to the incomparability graph of a natural unit interval order $P'=(m_1', m_2', \dots, m_{n-3}')$ with $m_1'=\cdots=m_{s-1}'=n-1$ and $m_s'=n-2$. By lemma \ref{lem:reflection}, $X_{{\rm inc}(P')}(\bold{x},t)=X_{{\rm inc}(\widetilde{P'})}(\bold{x},t)$ where $\widetilde{P'}=(\widetilde{m}_1', \widetilde{m}_2', \dots, \widetilde{m}_{n-3}')$ with $\widetilde{m}_1'=n-s-1$ and $\widetilde{m}_2'=n-2$. By Corollary \ref{cor:oneline}, 
$$X_{G-\{i,j\}}(\bold{x},t)=[n-4]_t!([n-2]_t[n-s-2]_t e_{n-2}+t^{n-s-2}[s-1]_t e_{(n-3,1)}).$$
Therefore,\\
$\displaystyle{\sum_{j=r+1}^{n-1} \sum_{T\in \widebar{\mathcal{T'}}_{P,(1,j)}} t^{{\rm inv}_{G}(T)}=t^{r-1}[n-r-1]_t [n-4]_t!([n-2]_t[n-s-2]_t e_{(n-1,1)}}$
$$\quad\qquad\qquad\qquad\qquad\qquad\qquad\qquad\qquad\qquad\qquad\qquad\qquad\qquad+t^{n-s-2}[s-1]_t e_{(n-2,2)}).$$
\end{enumerate}
Combining all of these, we have 
\begin{eqnarray*}
C_{(n-2,2)}(t)&=&[n-4]_t!(t^{n+r-s-2}[s-1]_t[n-r-1]_t+t^{n+r-s-3}[n-r-1]_t[s-1]_t)\\
&=&[n-4]_t!t^{n+r-s-3}[2]_t[n-r-1]_t [s-1]_t\\
&=&[n-4]_t!\widebar{C}_{(n-2,2)}(t),
\end{eqnarray*}
and
$$C_{(n-1,1)}(t)=[n-4]_t![n-2]_t(t^{n-s-1}[n-3]_t+t^{n-s}[r-2]_t[s-1]_t+t^{r-1}[n-r-1]_t[n-s-2]_t).$$
Using Lemma \ref{rem:uni}, we know that $C_{(n-2,2)}(t)$ is $e$-unimodal with center of symmetry $\frac{|E(G)|}{2}$.
 
Finally, to show that $C_{(n-1,1)}(t)$ is $e$-unimodal with center of symmetry $\frac{|E(G)|}{2}$, let 
$$C_{(n-1,1)}=[n-4]_t![n-2]_t A_n(r,s;t).$$
Let $\widetilde{P}=(\widetilde{m}_1, \widetilde{m}_2, \dots, \widetilde{m}_{n-1})$ be a natural unit interval order with $\widetilde{m}_1=n-s$, $\widetilde{m}_2=\cdots=\widetilde{m}_{n-r}=n-1$, and $\widetilde{m}_{n-r+1}=n$. For the incomparability graph $\widetilde{G}={\rm inc}(\widetilde{P})$,
by lemma \ref{lem:reflection}, $X_{G}(\bold{x},t)=X_{\widetilde{G}}(\bold{x},t)$.  Therefore,
\begin{eqnarray*}
A_{n}(r,s;t)&=&A_{n}(n-s,n-r;t)\\
&=&\frac{1}{2}\{A_{n}(r,s;t)+A_{n}(n-s,n-r;t)\}\\
&=&\frac{1}{2}\{(t^{n-s-1}[n-3]_t+t^{n-s}[r-2]_t[s-1]_t+t^{r-1}[n-r-1]_t[n-s-2]_t)\\
&~&\quad\qquad+(t^{r-1}[n-3]_t+t^{r}[n-s-2]_t[n-r-1]_t+t^{n-s-1}[s-1]_t[r-2]_t)\}\\
&=&\frac{1}{2}\{t^{r-1}(1+t^{n-r-s})[n-3]_t + t^{n-s-1}[2]_t[r-2]_t[s-1]_t\\ 
&~&\qquad\qquad\qquad\qquad\qquad\qquad\qquad\qquad\qquad+ t^{r-1}[2]_t [n-r-1]_t [n-s-2]_t\}.
\end{eqnarray*}
Since $t^{r-1}(1+t^{n-r-s})[n-3]_t$, $t^{n-s-1}[2]_t[r-2]_t[s-1]_t$, and $ t^{r-1}[2]_t [n-r-1]_t [n-s-2]_t$ are $e$-unimodal polynomials with center of symmetry $\frac{2n+r-s-6}{2}$, $A_{n}(r,s;t)$ is also $e$-unimodal with center of symmetry $\frac{2n+r-s-6}{2}$. By lemma \ref{lem:basic},
$$C_{(n-1,1)}(t)=[n-4]_t![n-2]_t A_n(r,s;t)=[n-4]_t!\widebar{C}_{(n-1,1)}(t)$$
is $e$-unimodal with center of symmetry $\frac{\binom{n-4}{2}+(n-3)+(2n+r-s-6)}{2}=\frac{|E(G)|}{2}$.
\end{proof}

\vspace{3mm}

\section*{Acknowledgements}\label{sec:acknow} The authors would like to thank Stephanie van Willigenburg for helpful conversations.

\vspace{3mm}
\bibliographystyle{unsrt}  
\bibliography{mybib} 
\end{document}